\documentclass{article}
\usepackage{amsmath}
\usepackage{mathrsfs} 
\usepackage{amsthm}
\usepackage{amsfonts}
\usepackage{amssymb}
\usepackage{latexsym}
\usepackage{tikz} 
\usepackage{esint} 
\usepackage{tikz}
\usepackage{tikz-cd}
\usepackage{cancel}
\usepackage{sectsty}
\usepackage{hyperref}
\usepackage{caption}
\usepackage{float}
\usepackage{setspace}
\usepackage{titlesec}
\titleformat{\subsection}{\normalfont\bfseries}{\thesubsection}{1em}{}

\usepackage{tocloft} 

\usepackage[OT2,T1]{fontenc}
\DeclareSymbolFont{cyrletters}{OT2}{wncyr}{m}{n}
\DeclareMathSymbol{\Sha}{\mathalpha}{cyrletters}{"58}

\sectionfont{\fontsize{12}{15}\selectfont}
\subsectionfont{\fontsize{11}{15}\selectfont}
\subsubsectionfont{\fontsize{10}{15}\selectfont}

\usepackage[matrix,tips,graph,curve]{xy}
\usepackage{enumitem}

\makeatletter
\usepackage[nospace, compress]{cite}

\usepackage{titlesec}
\titleformat{\subsection}[runin]{\normalfont\bfseries}{\thesubsection}{1em}{}
\titleformat{\subsubsection}[runin]{\normalfont\bfseries}{\thesubsubsection}{1em}{}

\usepackage{tocloft} 

\makeatother

\makeatletter 
\@addtoreset{equation}{section}
\makeatother

\theoremstyle{plain}
\newtheorem{theorem}[equation]{Theorem}

\newtheorem{lemma}[equation]{Lemma}
\newtheorem{proposition}[equation]{Proposition}
\newtheorem{conjecture}[equation]{Conjecture}

\theoremstyle{definition}
\newtheorem{definition}[equation]{Definition}

\newtheorem{remark}[equation]{Remark}

\newtheorem{set-up}[equation]{Set-up}

\newcommand{\IA}{\mathbb{A}}

\newcommand{\IC}{\mathbb{C}}
\newcommand{\ID}{\mathbb{D}}
\newcommand{\IE}{\mathbb{E}}
\newcommand{\IF}{\mathbb{F}}
\newcommand{\IG}{\mathbb{G}}

\newcommand{\IK}{\mathbb{K}}

\newcommand{\IQ}{\mathbb{Q}}
\newcommand{\IR}{\mathbb{R}}
\newcommand{\IS}{\mathbb{S}}

\newcommand{\IU}{\mathbb{U}}

\newcommand{\IZ}{\mathbb{Z}}

\newcommand{\sK}{\mathcal{K}}

\newcommand{\shG}{\mathscr{G}}

\newcommand{\End}{\mathrm{End}}
\newcommand{\tr}{\mathrm{tr}}

\newcommand{\Res}{\mathrm{Res}\,}

\renewcommand{\deg}{\mathrm{deg}}

\newcommand{\Spec}{\rm Spec\,}

\newcommand{\ch}{\rm ch} 

\newcommand\iso{\,{\cong}\,} 
\newcommand\tensor{{\otimes}}

\newcommand{\<}{\langle}
\renewcommand{\>}{\rangle}

\newcommand{\into}{\hookrightarrow}

\def\d/{/\mspace{-6.0mu}/}

\def\wt{\widetilde}

\def\what{\hat}

\newcommand{\sH}{\mathcal{H}}

\newcommand{\w}{\omega}

\newcommand{\Ohm}{\Omega}

\newcommand{\fl}{\mathrm{fl}}

\newcommand{\Pic}{\mathrm{Pic}\,}

\newcommand{\sM}{\mathcal{M}}
\newcommand{\sO}{\mathcal{O}}

\newcommand{\Cl}{\mathrm{Cl}}

\newcommand{\Gal}{\mathrm{Gal}}

\newcommand{\NS}{\mathrm{NS}}
\newcommand{\Hdg}{\mathrm{Hdg}}

\newcommand{\Isom}{\mathrm{Isom}}

\newcommand{\et}{\text{\'et}}

\newcommand{\Sh}{\mathrm{Sh}}
\newcommand{\shS}{\mathscr{S}}

\newcommand{\CSpin}{\mathrm{CSpin}}
\newcommand{\SO}{\mathrm{SO}}
\newcommand{\ad}{\mathrm{ad}}
\renewcommand{\sp}{\mathrm{sp}}
\newcommand{\GSp}{\mathrm{GSp}}
\newcommand{\GL}{\mathrm{GL}}

\newcommand{\sA}{\mathcal{A}}
\newcommand{\Br}{\mathrm{Br}}

\newcommand{\cris}{\mathrm{cris}}

\newcommand{\Mod}{\mathsf{M}}
\newcommand{\MT}{\mathrm{MT}}
\newcommand{\dR}{\mathrm{dR}}
\newcommand{\ab}{\mathrm{ab}}
\newcommand{\shA}{\mathscr{A}}
\newcommand{\ur}{\mathrm{ur}}

\newcommand{\an}{\mathrm{an}}
\newcommand{\Def}{\mathrm{Def}}

\newcommand{\F}{\mathrm{Fil}}

\newcommand{\Cris}{\mathrm{Cris}}
\newcommand{\Fr}{\mathrm{Fr}}

\newcommand{\Fil}{\mathrm{Fil}}
\newcommand{\BR}{\what{\mathrm{Br}}}

\newcommand{\Nm}{\mathrm{Nm}}

\newcommand{\sX}{\mathcal{X}}

\newcommand{\bH}{\mathbf{H}}
\newcommand{\bP}{\mathbf{P}}
\newcommand{\bL}{\mathbf{L}}
\newcommand{\isog}{\rightsquigarrow}

\newcommand{\shD}{\mathscr{D}}

\newcommand{\val}{\mathrm{val}}

\newcommand{\fq}{\mathfrak{q}}

\renewcommand{\Spec}{\mathrm{Spec\,}}

\newcommand{\bd}{\boldsymbol{\delta}}
\newcommand{\bdeta}{\boldsymbol{\beta}}
\newcommand{\bpi}{\boldsymbol{\pi}}

\newcommand{\bxi}{\boldsymbol{\xi}}
\newcommand{\fK}{\mathsf{K}}
\newcommand{\sD}{\mathcal{D}}

\begin{document}

\title{Isogenies between K3 Surfaces over $\bar{\IF}_p$}
\date{}
\author{Ziquan Yang}

\maketitle
\begin{abstract}
We generalize Mukai and Shafarevich's definitions of isogenies between K3 surfaces over $\IC$ to an arbitrary perfect field and describe how to construct isogenous K3 surfaces over $\bar{\IF}_p$ by prescribing linear algebraic data when $p$ is large. The main step is to show that isogenies between Kuga-Satake abelian varieties induce isogenies between K3 surfaces, in the context of integral models of Shimura varieties. As a byproduct, we show that every K3 surface of finite height admits a CM lifting under a mild assumption on $p$.
\end{abstract}

\maketitle
\section{Introduction}
Efforts to define the notion of isogeny between K3 surfaces have a long history. We refer the reader to \cite{Morrison} for a summary. Shafarevich defined an isogeny between two complex algebraic K3 surfaces $X, X'$ to be a Hodge isometry $H^2(X', \IQ) \stackrel{\sim}{\to} H^2(X, \IQ)$ (cf. \cite{Sha}). Mukai reserved the term for those Hodge isometries that are induced by correspondences (cf. \cite{Mukai}). Thanks to a recent result of Buskin \cite[Theorem 1.1]{Buskin}, now we know these two definitions coincide. Therefore, we make the following definition, which generalizes Mukai and Shafarevich's:
\begin{definition}
\label{isogdef} Let $k$ be a perfect field with algebraic closure $\bar{k}$. Let $X, X'$ be two K3 surfaces over $k$ with a correspondence (i.e., algebraic cycles on $X \times X'$ over $k$ with $\IQ$-coefficients) $f : X \isog X'$. 
\begin{itemize}
\item If $\mathrm{char\,} k = 0$, we say $f$ is an \textit{isogeny} over $k$ if the induced map $H^2_\et(X'_{\bar{k}}, \IA_f) \stackrel{\sim}{\to} H^2_\et(X_{\bar{k}}, \IA_f)$ is an isometry.
\item If $\mathrm{char\,} k = p$, we say $f$ is an \textit{isogeny} over $k$ if the induced maps $H^2_\et(X'_{\bar{k}}, \IA_f^p) \stackrel{\sim}{\to} H^2_\et(X_{\bar{k}}, \IA^p_f)$ and $H^2_\cris(X'/W(k))[1/p] \stackrel{\sim}{\to} H^2_\cris(X/W(k))[1/p]$ are isometries.
\end{itemize}
Two K3 surfaces $X$ and $X'$ over $k$ are said to be isogenous if there exists an isogeny between them. Two isogenies are viewed as equivalent if their cohomological realizations mentioned above all agree. 
\end{definition} 

Starting with a complex algebraic K3 surface, one can easily construct an isogenous one by prescribing a different $\IZ$ lattice in its rational Hodge structure. More precisely, using Buskin's result, the surjectivity of the period map and \cite[IV Theorem 6.2]{Barth}, one easily deduces the following:

\begin{theorem}
\label{teaser}
Let $X$ be a complex algebraic K3 surface. Let $\Lambda$ be a quadratic lattice isomorphic to $H^2(X, \IZ)$. Then for each isometric embedding $\Lambda \subset H^2(X, \IQ)$, there exists another complex algebraic K3 surface $X'$ together with an isogeny $f : X \isog X'$ such that $f^* H^2(X', \IZ) = \Lambda$. 
\end{theorem} 

We phrase and prove a partial analogue of Theorem~\ref{teaser} for quasi-polarized K3 surfaces in positive characteristic, using K3 crystals in the sense of Ogus and $\what{\IZ}^p$-lattices to replace $\IZ$-lattices. K3 crystals are F-crystals with properties modeled on the second crystalline cohomology of K3 surfaces (cf. \cite[Def.~3.1]{Ogus}):
\begin{definition}
 Let $k$ be a perfect field of characteristic $p > 0$ and let $\sigma$ be the lift of Frobenius on $W(k)$. A \textit{K3 crystal} over $k$ is a finitely generated self-dual quadratic lattice $D$ over $W(k)$ equipped with a $\sigma$-linear injection $\varphi : D \to D$ such that $p^2 D \subset \varphi(D)$, $\mathrm{rank\,} \varphi \tensor k = 1$ and $\varphi$ satisfies the following equation with the symmetric bilinear pairing $\<-, -\>$ on $D$: For any $x, y \in D$, $\< \varphi(x), \varphi(y)\> = p^2 \sigma(\< x, y \>)$. 
\end{definition}

Recall that a quasi-polarized K3 surface is a pair $(X, \xi)$ where $X$ is a K3 surface and $\xi$ is a big and nef line bundle on $X$. We call the self-intersection number of $\xi$ the degree of $(X, \xi)$. Customarily, we write $P^2_*(X) := \< \mathrm{ch_*}(\xi) \>^\perp \subset H^2_*(X)$ for the primitive cohomology of $(X, \xi)$ ($*= \dR, \cris, \et$, etc). Our main theorem states:

\begin{theorem}
\label{mainCor}
Let $(X, \xi)$ be a quasi-polarized K3 surface over $\bar{\IF}_p$ of degree $2d$. Assume that $p > 18d + 4$. Let $(M_p, \varphi)$ be a K3 crystal over $\bar{\IF}_p$ such that as a quadratic lattice $M_p$ is isomorphic to $P^2_\cris(X/W(\bar{\IF}_p))$ and let $M^p$ be a quadratic lattice isomorphic to $P^2_\et(X, \what{\IZ}^p)$. Then for each pair of isometric embeddings
\begin{enumerate}[label=(\roman*)]
\item $M_p \subset P^2_\cris(X/W(\bar{\IF}_p))[1/p]$ such that $\varphi$ agrees with the Frobenius action
\item $M^p \subset P^2_\et(X, \IA_f^p)$
\end{enumerate}
there exists another quasi-polarized K3 surface $(X', \xi')$ of degree $2d$ over $\bar{\IF}_p$ with an isogeny $f : X \isog X'$ such that $f^*\ch_*(\xi') = \ch_*(\xi)$ for $*= \cris, \et$, $f^* P^2_\cris(X'/W(\bar{\IF}_p)) = M_p$ and $ f^* P^2_\et(X', \what{\IZ}^p) = M^p$.
\end{theorem}

In section 6, we will propose a conjecture (cf. Conjecture~\ref{conj1}) which is an exact analogue of Thm~\ref{teaser} and explain how Thm~\ref{mainCor} is a partial result towards this conjecture (cf. Rmk~\ref{mainRemark}). Moreover, we show that this conjecture is completely known for supersingular K3 surfaces in odd characteristic (cf. Prop.~\ref{ss+}). Therefore, the paper is mainly concerned with the case when $X$ is not supersingular.

Our main idea is to apply Kisin's results on the Langlands-Rapoport conjecture \cite{Modp} to Shimura varieties which parametrize K3 surfaces, via the Kuga-Satake construction. We summarize the construction in the following diagram: 
\begin{center}
\begin{tikzcd}
\, & \shS(\CSpin(L_d), \Ohm) \arrow{d}{} \arrow{r}{} & \textit{A Siegel modular variety}\\\
\wt{\Mod}_{2d, \IZ_{(p)}}  \arrow{r}{} & \shS(\SO(L_d), \Ohm) & 
\end{tikzcd}
\end{center}
In this diagram, $L_d$ is a certain quadratic lattice; $\shS(\CSpin(L_d), \Ohm)$ and $\shS(\SO(L_d), \Ohm) $ are the canonical integral models of the Shimura varieties associated to algebraic groups $\CSpin(L_d)$, $\SO(L_d)$ and a certain period domain $\Ohm$; $\wt{\Mod}_{2d, \IZ_{(p)}}$ is a moduli stack of primitively quasi-polarized K3 surfaces of degree $2d$ with some additional structures. We have temporarily suppressed level structures, but we will introduce these objects in detail in Section 3. The work of Kisin clarifies the notion of isogeny classes on $\shS(\CSpin(L_d), \Ohm)(\bar{\IF}_p)$. We can prove that, as one would expect, K3 surfaces with isogenous Kuga-Satake abelian varieties are themselves isogenous. 
\begin{proposition} 
\label{mainTh}
Assume $p \nmid d$ and $p \ge 5$. Let $t, t' \in \wt{\Mod}_{2d, \IZ_{(p)}}(\bar{\IF}_p)$ correspond to quasi-polarized K3 surfaces $(X_t, \bxi_t)$ and $(X_{t'}, \bxi_{t'})$ of degree $2d$. If the images of the points $t, t'$ in $\shS(\SO(L_d), \Ohm)(\bar{\IF}_p)$ lift to points $s, s' \in \shS(\CSpin(L_d), \Ohm)(\bar{\IF}_p)$ in the same isogeny class, then there exists an isogeny $X_t \isog X_{t'}$ such that $f^*\ch_*(\bxi_{t'}) = \ch_*(\bxi_t)$ for $* = \cris, \et$.
\end{proposition}
The precise form of this assertation will be stated in Prop.~\ref{main}. Once we have Prop.~\ref{main}, we can prove Theorem~\ref{mainCor} by a group-theoretic computation. The condition on $p$ arises when we apply the surjectivity of the period map, due to Matsumoto (cf. \cite[Thm~4.1]{Matsumoto}). 

As a key intermediate step in \cite{Modp}, Kisin proved that in each isogeny class of mod $p$ points of a Shimura variety of Hodge type, there exists a point which lifts to a special point. Prop.~\ref{mainTh} implies that, under mild assumptions, K3 surfaces over $\bar{\IF}_p$ also admit CM liftings up to isogeny. 
It turns out that for K3 surfaces of finite height, one can propagate the property of having CM lifting within an isogeny class, so there is no need to pass to isogeny. 
\begin{theorem}
\label{CML}
Let $X/\bar{\IF}_p$ be a K3 surface of finite height. Suppose $X$ admits a quasi-polarization of degree $p \nmid d$ and $p \ge 5$. Then there exists a finite extension $V/W(\bar{\IF}_p)$ and a lift $X_V$ of $X$ to $V$ such that for any $V \into \IC$, the complex K3 surface $X_V \times \IC$ has commutative Mumford-Tate group. Moreover, the natural map $\Pic(X_V) \to \Pic(X)$ is an isomorphism. 
\end{theorem}
While preparing this paper, the author noticed that K. Ito, T. Ito and T. Koshikawa released a new preprint \cite{IIK} which contains the above result with no assumption on $p$.

Finally, we remark that isogenies between K3 surfaces (in our sense) can also be given by constructing moduli of twisted sheaves. The moduli theory of twisted sheaves for complex K3's was initiated by S. Mukai and its generalization to positive characteristic has been studied by Lieblich, Maulik, Olsson and Snowden (cf. \cite{LO1}, \cite{LO2}, \cite{LMS}). Huybrechts has shown that in fact, over $\IC$, \textit{every isogeny} can be realized by a sequence of such operations \cite[Thm~0.1]{Huy}. It would be interesting to figure out whether this is true over $\bar{\IF}_p$. 

\paragraph{Outline/Strategy of the Paper} In section 2 and 3, we review the Kuga-Satake construction and the relevant Shimura varieties, mostly following \cite{Keerthi} and \cite{CSpin}. In addition, we emphasize that an isogeny between Kuga-Satake abelian varieties should not be a bare isogeny between abelian varieties, but one which repects certain tensors. We call such isogenies ``CSpin-isogenies". This is an analogue of Andr\'e's notion of an isomorphism between Kuga-Satake packages (cf. \cite[Def.~4.5.1, 4.7.1]{Andre}). We also discuss \textit{special endomorphisms} of Kuga-Satake abelian varieties, which correspond to line bundles on K3 surfaces. Special endomorphisms are preserved by CSpin-isogenies. 

In Section 4, we prove a lifting lemma, which we use in Section 5 to prove the theorems of the paper. The main idea is that a K3 surface of finite height deforms like an elliptic curve: There is a natural one-dimensional formal group of finite height attached to each such surface $X$, namely the formal Brauer group $\BR_X$ of $X$. Moreover, there is a natural map of deformation functors 
$ \Def_X \to \Def_{\BR_X} $.
Although this map is not an isomorphism, Nygaard and Ogus constructed a natural section to this map. We make use of this section to show that away from the supersingular locus, CSpin-isogenies always lift to characteristic zero. This treats the finite height case for Prop.~\ref{mainTh}. We treat the supersingular case separately.

In Section 6, we discuss how isogenies interact with ample cones, polarizations, and quasi-polarizations, and how Torelli theorems can be understood from the perspective of isogenies. Then we will discuss how Thm~\ref{mainCor} is related to an exact analogue of Thm~\ref{teaser}. 
\paragraph{Notations and Conventions}
\begin{itemize}
    \item $\what{\IZ}$ denotes the profinite completion of $\what{\IZ}$, and $\what{\IZ}^p$ denotes the prime-to-$p$ part of $\what{\IZ}$. We write $\IA_f$ for the finite adeles and $\IA^p_f$ for the prime-to-$p$ part. When $k$ is a perfect field of characteristic $p > 0$, we denote by $W(k)$ the Witt ring of vectors. The lift of Frobenius on $W(k)$ is denoted by $\sigma$. If $X$ is a smooth projective variety over $k$, the $\sigma$-linear Frobenius action on $H^2_\cris(X/W(k))$ is denoted by $F$. 
    \item Suppose $k$ is a perfect field of characteristic $p > 0$ and $S$ is a scheme on which $p$ is locally nilpotent over $W(k)$. 
    For any $p$-divisible group $\shG$ over $S$, we let $\ID(\shG)$ denote its contravariant Dieudonn\'e crystal on $\Cris(S/W(k))$ and $\shG^*$ denote its Cartier dual. When $S = \mathrm{Spec\,} k$, we also write the $W(k)$-module $\ID(\shG)_{W(k)}$ simply as $\ID(\shG)$.
    \item For a ring $R$ and a finite free $R$-module $M$, we denote by $M^\tensor$ the direct sum of the $R$-modules obtained from $M$ by taking duals, tensor products, symmetric and exterior powers.
    \item For any field extension $E/F$ and algebraic group $G$ over $E$, we write $\Res_{E/F} G$ for the Weil restriction of $G$ to $F$. $\IS := \Res_{\IC/\IR} \IG_m$ denotes the Deligne torus.
    \item By a quadratic lattice over a ring $R$ we always mean a finitely generated free module over $R$ equipped with a symmetric bilinear pairing. 
\end{itemize}

\section{Hodge Theoretic Preparations}
\paragraph{Hodge Structures of K3 Surfaces} 
We collect some definitions and facts about Hodge structures of K3 type. The definition below differs from the usual definition (cf. \cite[Ch.~3 Def.~2.3]{Huy}) by a Tate twist.

\begin{definition}
A rational or integral Hodge structure $V$ of weight $0$ is said to be of \textit{K3 type} if $\dim V^{-1,1} = 1$ and $V^{a, b} = 0$ if $|a - b| > 2$.
\end{definition}

Let $X$ be a complex K3 surface. The \textit{transcendental lattice} of $X$, denoted by $T(X)$, is $\Pic(X)^\perp \subset H^2(X, \IZ(1))$. $H^2(X, \IZ(1))$ and $T(X)$ are both Hodge structures of K3 type. Moreover, if $X$ is algebraic, then $T(X)_\IQ := T(X) \tensor \IQ$ is polarizable and irreducible as a Hodge structure (cf. \cite[Lemma 3.1]{K3book}). We denote by $\MT(-)$ the Mumford-Tate group associated to a Hodge structure.

\begin{theorem} Let $V$ be a rational polarized irreducible Hodge structure which is of K3 type. 
\label{Zarhin}
\begin{enumerate}[label = (\alph*)]
\item $E:= \End_{\Hdg}(V)$ is either a totally real field or a CM field. 
\item $\MT(V)$ is commutative if and only if $\dim_E V = 1$. In this case, $E$ is a CM field and $\MT(V) \subset \GL(V)$ is equal to 
\begin{align}
\label{maximalTorus}
    \ker (\mathrm{Nm}: \Res_{E/\IQ} \IG_m \to \Res_{F/\IQ} \IG_m)
\end{align}
where $F$ is the maximal totally real subfield of $E$ and $\mathrm{Nm}$ is the norm map. Note that by definition $E$ is a subalgebra of $\End(V)$, so there is an embedding $\Res_{E/\IQ} \IG_m \subset \GL(V)$. 
\item When $E$ is a CM field, there exists a Hodge isometry $\tau$ such that $E = \IQ(\tau)$.
\end{enumerate}
 \end{theorem}
\begin{proof}
(a) and (b) are due to Zarhin (cf. \cite[Thm~1.5.1, Rmk~1.5.3, Thm~2.3.1]{Zarhin}). (c) can be found in \cite[Thm~3.7]{K3book}, where the result is credited to Borcea \cite{Borcea}.
\end{proof}

 \begin{remark}
$X$ is said to have CM if $\MT(T(X)_\IQ)$ is commutative. Some authors say that $X$ has CM if $E$ is a CM field. We emphasize that $E$ being a CM field is not enough to ensure that $\MT(T(X)_\IQ)$ is commutative. Here is an example: Take two elliptic curves $E_1$ and $E_2$. Suppose $E_1$ has CM but $E_2$ does not. Apply the Kummer construction to the abelian surface $E_1 \times E_2$. We obtain a K3 surface $X$ with $\dim T(X)_\IQ = 4$ but $E := \End_{\Hdg}(T(X)_\IQ)$ is a CM field of dimension $2$.  The real points of the Hodge group $\Hdg(T(X)_\IQ)$ form a unitary group of signature $(1+, 1-)$ which is not commutative.  
\end{remark}

\label{KS}
\paragraph{A Review of the Clifford Algebra}
Let $L$ be a self-dual quadratic lattice over a commutative ring $R$. Let $q$ denote the quadratic form attached to $L$. Assume that $2$ is invertible in $R$. We can form the Clifford algebra $$\Cl(L) := (\oplus_{n \ge 0} L^{\tensor n})/(v \tensor v - q(v)). $$
$\Cl(L)$ has a natural $\IZ/2\IZ$-grading and we denote by $\Cl^+(L)$ (resp. $\Cl^-(L)$) the even (resp. odd) part. We define the group $\CSpin(L)$ by 
$$ \CSpin(L) = \{ v \in \Cl^+(L)^\times : v L v^{-1} = L \}. $$

We give $\CSpin(L)$ the structure of an algebraic group over $R$ by defining $\CSpin(L)(R') = \CSpin(L_{R'})$ for any $R$-algebra $R'$. There is a norm map $\mathrm{Nm} : \CSpin(L) \to \IG_m$ given by $v \mapsto v^* \cdot v$, where $v^*$ is the natural anti-involution of $v$. By letting $\CSpin(L)$ act on $L$ by conjugation, we get the adjoint representation of $\CSpin(L)$, which fits into an exact sequence of algebraic groups
\begin{equation}
\label{coreExt}
1 \to \IG_m \to \CSpin(L) \stackrel{\ad}{\to} \SO(L) \to 1.
\end{equation} 
Let $H$ be a free $\Cl(L)$ bi-module of rank $1$. Note that $H$ has a natural $\IZ/2 \IZ$-grading. Left multiplication of $\Cl(L)$ on $H$ gives us a spin representation $\sp : \CSpin(L) \to \GL(H)$ and an embedding 
\begin{equation}
\label{embeddingofL}
L \into \Cl(L) \into \End (H)
\end{equation} 
Equip $\End (H)$ with the metric $(\alpha, \beta) := 2^{-\mathrm{rank\,} L} \tr(\alpha \circ \beta)$. Then we have
\begin{lemma}
\label{stabilizer}
\begin{enumerate}[label=(\alph*)]
\item The embedding (\ref{embeddingofL}) is an isometry with respect to $q$ and $( -, - )$.
\item There exists a unique orthogonal projection $\pi : \End (H) \to L$.
\item $\CSpin(L)$ is the stabilizer of $\pi \in H^{(2, 2)}$ among all automorphisms of $H$ as a $\IZ/2\IZ$-graded right $\Cl(L)$-module. 
\end{enumerate}
\end{lemma}  
\begin{proof}
See \cite[1.3]{CSpin}. 
\end{proof}

\paragraph{The Kuga-Satake Construction} Take $R = \IQ$ and assume that $(L, q)$ has signature $(n+, 2-)$. The Kuga-Satake construction associates a Hodge structure $\wt{h}$ on $H$ of type $\{(1,0), (0, 1)\}$ to each Hodge structure $h$ on $L$ of type $\{(-1, 1), (0, 0), (1, -1)\}$ which respects $q$. It can be packed in the following diagram 
\begin{figure}[H]
\centering
\begin{tikzcd}
\, & \CSpin(L_\IR) \arrow{d}{\ad} \arrow{r}{\sp} & \GL(H_\IR)\\
\IS \arrow{ur}{\wt{h}} \arrow{r}{h} & \SO(L_\IR ) & 
\end{tikzcd}
\caption{Kuga-Satake construction for a Hodge structure of K3 type}
\end{figure} 
For each $h$, there is a unique lift $\wt{h}$ such that $\mathrm{Nm} \circ \wt{h}$ gives the norm map $\IS \to \IG_{m, \IR}$. The composition $\sp \circ \wt{h}$ gives the desired Hodge structure of weight $1$ on $H$ (cf. \cite[4.2]{Deligne1}). In fact, the induced Hodge structure on $H$ has a simple description by that on $L$:
$$ \Fil^1 H_\IC = \ker (\Fil^1 L_\IC) $$
Here we are viewing $L$ as a subspace of $\End (H)$ via left multiplication. Since $\Fil^1 L_\IC$ is one-dimensional, it makes sense to talk about $\ker (\Fil^1 L_\IC)$.

The following simple observation makes the Kuga-Satake construction naturally suitable for studying complex multiplication: 
\begin{lemma}
\label{CMcors}
$\MT(\wt{h})$ is commutative if and only if $\MT(h)$ is commutative.  
\end{lemma}
\begin{proof}
We obtain the following exact sequence 
$$ 1 \to \IG_{m, \IQ} \to \MT(\wt{h}) \to \MT(h) \to 1 $$
by pulling back (\ref{coreExt}) along the inclusion $\MT(h) \into \SO(L)$. If $\MT(h)$ is commutative, then by Thm~\ref{Zarhin}(b) it is a torus. One easily deduces from \cite[IV Cor.~11.5]{LAG} that an extension of a torus by another torus is again a torus, so $\MT(\wt{h})$ is commutative. The converse implication is clear. 
\end{proof}

\begin{lemma}
\label{surjection}
Every isometry $g \in \SO(L)(\IQ)$ preserving the Hodge structure $h$ lifts to $\wt{g} \in \CSpin(L)(\IQ)$ preserving $\wt{h}$. In other words, the natural map of centralizers $$C_{\CSpin(L)}(\MT(\wt{h}))(\IQ) \to C_{\SO(L)}(\MT(h))(\IQ)$$ 
is surjective. 
\end{lemma}
\begin{proof}
Let $\wt{g} \in \CSpin(L)(\IQ)$ be any lift of $g \in \SO(L)(\IQ)$. In fact, since the kernel of $\ad$ lies in the center of $\CSpin(L)$, one lift works if and only if any other one does. Let $\w$ be any generator of $\Fil^1 L_\IC$. Since $g$ preserves the Hodge structure on $L$ and $\Fil^1 L_\IC$ is one-dimensional, we have $ \wt{g} \w \wt{g}^{-1} = \lambda \w$, or equivalently $\wt{g} \w = \lambda \w \wt{g}, \text{ for some } \lambda \neq 0 \in \IC$. Now if $a \in \Fil^1 H_\IC$, then $\w \wt{g} a = \lambda^{-1} \wt{g} \w a = 0$, which implies that $\wt{g} a \in \Fil^1 H_\IC$. 
\end{proof}

\section{The Kuga-Satake Period Map}
\subsection{Shimura Varieties}
\label{Shimura}
In this section, we review the theory of spinor and orthogonal Shimura varieties that we need. 

\subsubsection{} Let $L$ be a quadratic lattice of signature $((m-2)+, 2-)$ with $m \ge 3$ over $\IZ$. Let $q$ be the quadratic form attached to $L$. Let $p > 2$ be a prime. Assume that $L \tensor \IZ_{(p)}$ is self-dual. Let $G$ (resp. $G^\ad$) denote the reductive group over $\IZ_{(p)}$ given by $\CSpin(L \tensor \IZ_{(p)})$ (resp. $\SO(L \tensor \IZ_{(p)})$). Again let $H$ denote $\Cl(L)$, viewed as a $\Cl(L)$-bimodule. For every compact open subgroup $U$ of $G(\IA_f)$, $G(\IA_f^p)$, or $G(\IQ_p)$, we denote by $U^\ad$ the image of $U$ under the adjoint map $G \to G^\ad$. Set $\fK_p = G(\IZ_p)$. Let $\fK_L \subset G(\IA_f)$ be the intersection $G(\IA_f) \cap \Cl(L \tensor \what{\IZ})^\times$. Then $\fK_L$ (resp. $\fK_L^\ad$) is of the form $\fK_p \fK_{L}^p$ (resp. $\fK^\ad_p \fK_L^{p, \ad}$), where $\fK_L^p$ (resp. $\fK_L^{p, \ad}$) is a compact open subgroup of $G(\IA_f^p)$ (resp. $G^\ad(\IA_f^p)$).

We fix a choice of a non-zero element $\delta \in \det(L) = \wedge^m L$. Since there is a natural embedding $L \into \End(H)$ given by left multiplication, we can also view $\delta$ as an element of $\wedge^m H^{\tensor (1, 1)}$. 

Let $\Ohm$ be the space of oriented negative definite planes in $L_\IR$. Then the pairs $(G_\IQ, \Ohm)$ and $(G^\ad_{\IQ}, \Ohm)$ define Shimura data with reflex field $\IQ$ (cf. \cite[3.1]{CSpin}). We can choose a symplectic form $\psi : H \times H \to \IZ$ such that the embedding $G \into \GL(H)$ induced by the spin representation factors through the general symplectic group $\GSp := \GSp(H, \psi)$ (cf. \cite[3.5]{CSpin}). Let $\sH$ be the Siegel half-spaces attached to $(H, \psi)$. Then we have an embedding of Shimura data $(G, \Ohm) \into (\GSp, \sH)$. 

Let $\sK_H \subset \GSp(\IA_f)$ be the maximal compact open subgroup which stabilizes $H_{\what{\IZ}}$. Then $\sK_H$ is of the form $\sK_{p} \sK_{H}^p$ with $\sK_p \subset \GSp(\IQ_p)$ and $\sK_{H}^p \subset \GSp(\IA^p_f)$. We have that $\fK_p = \sK_p \cap G(\IQ_p)$.

\subsubsection{}
\label{GSpMod} We introduce the moduli interpretation of the Shimura variety $\Sh_{\sK}(\GSp, \sH)$ (cf. \cite[3.7]{CSpin}, see also \cite[1.3.4]{Modp}), assuming that $\sK$ is of the form $\sK_p \sK^p$, where $\sK_p$ is the $p$-part of $\sK_H$ as above and $\sK^p \subset \GSp(\IA^p_f)$ is chosen to be sufficiently small. Let $T$ be a $\IZ_{(p)}$-scheme. Let $\shA(T)$ denote the category of abelian schemes over $T$ and let $\shA(T) \tensor \IZ_{(p)}$ be the category obtained by tensoring the Hom groups in $\shA(T)$ by $\IZ_{(p)}$. An object in $\shA(T) \tensor \IZ_{(p)}$ is called an abelian sheme up to prime-to-$p$ isogeny over $T$. An isomorphism in $\shA(T) \tensor \IZ_{(p)}$ is called a $p'$-quasi-isogeny. Now let $B$ be an object in $\shA(T) \tensor \IZ_{(p)}$ and let $B^*$ be its dual. By a weak polarization on $B$ we mean an equivalence class of $p'$-quasi-isogenies $\lambda : B \to B^*$ such that some multiple of $\lambda$ is a polarization and two such $\lambda$'s are equivalent if they differ by an element in $\IZ_{(p)}^\times$. Let $f : B \to T$ be the structure morphism. For the pair $(B, \lambda)$, denote by $\underline{\Isom}(H \tensor \IA^p_f, R^1 f_{\et *} \underline{\IA}_f^p)$ the set of isomorphisms $\underline{H \tensor \IA_f^p} \stackrel{\sim}{\to} R^1 f_{\et *} \underline{\IA}_f^p$ which are compatible with the pairings induced by $\psi$ and $\lambda$ up to a $(\IA^p_f)^\times$ scalar. Here $\underline{H \tensor \IA_f^p}$ denotes the constant sheaf on $T$ with coefficients $H \tensor \IA_f^p$. A $\sK$-level structure on $(B, T)$ is a section $\varepsilon_{\sK}^p \in H^0(T, \underline{\Isom}(H \tensor \IA^p_f, R^1 f_{\et *} \underline{\IA}_f^p)/\sK^p).$ The functor which sends each $T$ over $\IZ_{(p)}$ to the set of isomorphism classes of triples $(B, \lambda, \varepsilon_{\sK}^p)$ as above is representable by a scheme $\shS_{\sK}(\GSp, \sH)$ over $\IZ_{(p)}$, whose fiber of $\IQ$ is naturally identified with $\Sh_{\sK}(\GSp, \sH)$. 

The above description of points on $\shS_{\sK}(\GSp, \sH)$ or $\Sh_{\sK}(\GSp, \sH)$ depends only on the $\IZ_{(p)}$-lattice $H_{\IZ_{(p)}}$. Since we have a chosen $\IZ$-lattice $H$ inside $H_{\IZ_{(p)}}$ and $\psi$ induces an embedding $H \into H^\vee$, there is a universal abelian scheme over $\shS_{\sK}(\GSp, \sH)$ (cf. \cite[2.3.3]{int}). More precisely, if $d'$ is the index of $H$ in $H^\vee$ under the embedding induced by $\psi$, then for each triple $(B, \lambda, \varepsilon^p_\sK)$ as in the previous paragraph, there is a unique triple of the form $(A \stackrel{q}{\to} T, \mu, \epsilon^p_\sK)$ where $A$ is an abelian scheme over $T$, $\mu$ is a polarization on $A$ of degree $d'$, and $ \epsilon_\sK^p \in H^0(T, \underline{\mathrm{Isom}}(H \tensor \what{\IZ}^p, R^1 q_{\et *} \underline{\what{\IZ}}^p)/ \sK^p)$ such that when we view $A$ only as an abelian scheme up to prime to $p$ isogeny and $\mu$ as a weak polarization on $\lambda$, the triple $(A, \mu, \epsilon_\sK^p \tensor \IQ)$ is isomorphic to $(B, \lambda, \varepsilon_\sK^p)$. Here the (pro)-\'etale sheaf $\underline{\mathrm{Isom}}(H \tensor \what{\IZ}^p, R^1 q_{\et *} \underline{\what{\IZ}^p})$ on $T$ is defined to be the set of the isomorphisms $\underline{H \tensor \what{\IZ}^p} \stackrel{\sim}{\to} R^1 f_{\et *} \underline{\what{\IZ}}^p$ which are compatible with the pairings induced by $\psi$ and $\lambda$ up to a $(\what{\IZ}^p)^\times$ scalar. Note that since we have assumed that $\sK^p$ is sufficiently small, the triple $(B, \lambda, \varepsilon^p_\sK)$, and hence $(A, \mu, \epsilon^p_\sK)$, have no nontrivial automorphisms.

By \cite[Lem.~2.1.2]{int}, for any choice of compact open subgroup $\fK^p \subset G(\IA^p_f)$, there exists a compact open subgroup $\sK^p \subset \GSp(\IA_f^p)$ which contains $\fK^p$ such that for $\sK := \sK_p \sK^p$ and $\fK := \fK_p \fK^p$, the embedding of Shimura data $(G, \Ohm) \into (\GSp, \sH)$ induces an embedding $\Sh_{\fK}(G, \Ohm) \into \Sh_{\sK}(\GSp, \sH)$. When $\fK^p$ and $\sK^p$ are chosen to be sufficiently small, the canonical integral model $\shS_{\fK}(G, \Ohm)$ of $\Sh_{\fK}(G, \Ohm)$ over $\IZ_{(p)}$ is contructed by taking the normalization of the closure of $\Sh_{\fK}(G, \Ohm)$ inside $\shS_{\sK}(\GSp, \sH)$ (cf. \cite[Thm~2.3.8]{int}), and the canonical integral model $\shS_{\fK^\ad}(G^\ad, \Ohm)$ of $\Sh_{\fK^\ad}(G^\ad, \Ohm)$ over $\IZ_{(p)}$ can be constructed out of $\shS_{\fK}(G, \Ohm)$ as a quotient (cf. \cite[2.3.9]{int}, see also \cite[Thm~4.4]{CSpin}). What is important for us is that the natural morphism $\Sh_{\fK}(G, \Ohm) \to \Sh_{\fK^\ad}(G^\ad, \Ohm)$ is a finite \'etale cover which is a torsor of the group (cf. \cite[(3.2)]{CSpin})\footnote{We have communicated to Madapusi Pera and confirmed that in \cite[(3.2)]{CSpin} there should be no ``$>0$" sign in the formula of  $\Delta(K)$. }
\begin{equation}
\label{Delta}
\Delta(K) = \IA^\times_f / \IQ^\times (K \cap \IA^\times_f) = \IA^{p, \times}_f / \IZ_{(p)}^\times (\fK^p \cap \IA^{p, \times}_f)
\end{equation} 
and this morphism extends to a pro-\'etale cover $\shS_{\fK}(G, \Ohm) \to \shS_{\fK^\ad}(G^\ad, \Ohm)$ which is a torsor of the same group. 

\subsubsection{} Next, we introduce the sheaves on spinor or orthogonal Shimura varieties. To begin with, the spin representation $\CSpin(L) \to \GL(H)$ (resp. adjoint representation $\CSpin(L) \to \SO(L)$) gives rise to a variation of $\IZ$-Hodge structures $(\bH_B, \Fil^\bullet \bH_{\dR, \IC})$ (resp. $(\bL_B, \Fil^\bullet \bL_{\dR, \IC})$) on $\Sh_{\fK}(G, \Ohm)_\IC$ (cf. \cite[3.4]{CSpin}). $(\bH_B, \Fil^\bullet \bH_{\dR, \IC})$ carries in addition a $\IZ/2\IZ$-grading, right $\Cl(L)$-action and a tensor $\bpi_B \in H^0(\Sh_K(G, \Ohm)_\IC, \bH_B^{\tensor (2, 2)} \tensor \IQ)$ such that $(\bL_B \tensor \IQ, \Fil^\bullet \bL_{\dR, \IC})$ can be identified with the image $\bpi_B((\bH_B \tensor \IQ)^{\tensor (1, 1)}, \Fil^\bullet \bH_{\dR, \IC})$ and $\bL_B^\vee \subset \bL_B\tensor \IQ$ is identified with $\bpi_B(\bH_B^{\tensor (1, 1)})$ (cf. \cite[1.5]{CSpin}). Moreover, since $\CSpin(L)$ stabilizes $\delta \in \det(L) \subset (\wedge^m H^{\tensor (1, 1)})$, we have a global section $\bd_B \in H^0(\Sh_{\fK}(G, \Ohm)_\IC, \det(\bL_B))$. For the following, assume that $\fK = \fK_p \fK^p$ with $\fK^p \subset \fK_L^{p}$ sufficiently small.

Now let $h : \sA \to \shS_{\fK}(G, \Ohm)$ be the pullback of the universal abelian scheme over $\shS_\sK(\GSp, \sH)$. Let $h_\IQ$ (resp. $h_\IC$, resp. $\bar{h}$) denote the fiber of $h$ over $\IQ$ (resp. $\IC$, resp. $\IF_p$). Let $\bH_\dR$ be the first relative de Rham cohomology of $\sA$ over $\shS_{\fK}(G, \Ohm)$. For every rational prime $\ell$, write $\bH_{\ell, \IQ}$ for $R^1 h_{\IQ, \et *} \underline{\IZ}_\ell$ and for $\ell \neq p$ set $\bH_\ell = R^1 h_{\et *} \underline{\IZ}_\ell$. Let $h_\IC^\an$ be the morphism between complex analytic manifolds associated to $h_\IC$. The variation of $\IZ$-Hodge structures given by the first relative Betti and de Rham cohomology of $h_\IC^\an$ can be identified with the aforementioned $(\bH_B, \Fil^\bullet \bH_{\dR, \IC})$. Let $\bH_\cris$ denote the crystal of vector bundles on the the crystalline site $\Cris(\shS_{\fK}(G, \Ohm)_{\IF_p}/ \IZ_p)$ given by $R^1 \bar{h}_{\cris *} \sO_{\sA \tensor \IF_p / \IZ_p}$.

Since the variation of $\IZ$-Hodge structures $(\bH_B, \Fil^\bullet \bH_{\dR, \IC})$ carries a $\IZ/2 \IZ$-grading and a right $\Cl(L)$-action, the abelian scheme $\sA_\IC$ over $\Sh_{\fK}(G, \Ohm)_\IC$ carries a $\IZ/2 \IZ$-grading and a left $\Cl(L)$-action. The global section $\bpi_B$ gives rise to a global section $\bpi_{\dR, \IC}$ of the vector bundle $\Fil^0 \bH_{\dR, \IC}^{\tensor (2, 2)}$ and a global section $\bpi_{\ell, \IC}$ of $\bH^{\tensor (2, 2)}_{\ell, \IC}$, where $\bH_{\ell, \IC}$ is the \'etale local system over $\Sh_\fK(G, \Ohm)_\IC$ obtained by restricting $\bH_\ell$. Similarly, $\bd_B$ gives rise a global section $\bd_{\ell, \IC}$ of $\wedge^m \bH^{\tensor (1, 1)}_{\ell, \IC}$. 

Now we explain how to equip $\sA$ with ``CSpin-structures''. Results here are taken from \cite[(2.2)]{int} and \cite[Prop. 3.11, Sect.~4.5]{CSpin}. 
The $\IZ/2 \IZ$-grading and the left $\Cl(L)$-action on $\sA_\IC$ descend to $\sA_\IQ$ and extend to $\sA$. The global sections $\bpi_{\ell, \IC}$ and $\bd_{\ell, \IC}$ descend to global sections $\bpi_{\ell, \IQ}$ and $\bd_{\ell, \IQ}$ of $\bH^\tensor_{\ell, \IQ}$ and if $\ell \neq p$, extend to global sections $\bpi_{\ell}$ and $\bd_{\ell}$ of $\bH_\ell^\tensor$. Similarly, the de Rham section $\bpi_{\dR, \IC}$ descend to a global section $\bpi_{\dR, \IQ}$ of $\bH_{\dR, \IQ}^\tensor := \bH_{\dR}^\tensor |_{\Sh_\fK(G, \Ohm)}$ and extend to a global section $\bpi_\dR$ of $\bH_\dR^\tensor$ (cf. \cite[Cor.~2.3.9]{int}). Finally, we introduce the crystalline realization of $\pi$. Let $\nabla$ be the Gauss-Manin connection on $\bH_\dR$. Since $\shS_{\fK}(G,\Ohm)$ is smooth, the crystal of vector bundles $\bH_\cris$ is completely determined by the restriction of the pair $(\bH_\dR, \nabla)$ to the completion of $\shS_{\fK}(G, \Ohm)$ along the special fiber. Since the global section $\bpi_\dR$ of $\bH_\dR^{\tensor (2, 2)}$ is horizontal with respect to $\nabla$, it gives rise to a global section $\bpi_\cris$ of $\bH_\cris$ (cf. \cite[4.14]{CSpin}).

\begin{remark}
\label{dRham}
The following can be extracted from the proof of \cite[2.3.5, 2.3.9]{int}, or \cite[Prop.~4.7]{CSpin}:
Let $k$ be a perfect field of characteristic $p$. Let $K$ be a totally ramified extension of $W(k)[1/p]$, $\sO_K$ be the ring of integers of $K$ and $\bar{K}$ be an algebraic closure of $K$. Let $s \in \shS_{\fK}(G, \Ohm)(k)$ be a point. Let $s_K$ be a $K$-valued point on $\shS_{\fK}(G, \Ohm)$ and $\bar{s}_K$ be the geometric point over $s_K$ associated to $\bar{K}$. Suppose that $s_K$ specializes to $s$ via an $\sO_K$-valued point. Under the de Rham comparison isomorphism 
$$ H^1_\dR(\sA_{s}/K) \tensor_K B_\dR \stackrel{\sim}{\to} H^1_\et(\sA_{s_K} \tensor \bar{K}, \IQ_p) \tensor_{\IQ_p} B_\dR $$
$\pi_{\dR, s_K}$ is sent to $\pi_{p, \bar{s}_K}$ by a result of Blasius \cite[Thm~0.3]{Blasius} and Wintenberger.  $\bpi_{\cris, s}$ is sent to $\bpi_{\dR, s_K}$ via the Berthelot-Ogus isomorphism $H^1_\cris(\sA_s/W(k)) \tensor_{W(k)} K \iso H^1_\dR(\sA_{s_K}/K)$.

In \textit{loc. cit.}, Blasius restricted to considering abelian varieties which can be obtained via base change from the ones defined over $\bar{\IQ}$, but this condition can be removed (cf. \cite[Thm~5.6.3]{Moonen}). The original reference for the de Rham comparison isomorphism in \cite{Blasius} is \cite[Thm~8.1]{Fal}. It is remarked in \cite[Sect.~11.3]{IIK} that for Blasius' theoerem one can alternatively use the de Rham comparison isomorphism constructed by Scholze (cf. \cite[Thm~8.4]{Scholze}, see also \cite[Sect.~11.1]{IIK}). For our purposes, all that is important is that $\bpi_{\dR, s_K}$, and hence $\bpi_{\cris, s}$, are completely determined by $\bpi_{p, \bar{s}_K}$. 
\end{remark}

The $\IZ/2\IZ$-grading, $\Cl(L)$-action and various realizations of the tensor $\pi$ are the ``CSpin-structure'' on $\sA$. Note that we do not think of realizations of $\delta$ as part of the ``CSpin-structure'' because $\delta$ is uncessary in cutting out $\CSpin(L)$ inside $\GL(H)$. Later will only make use of $\bd_B$ and $\bd_{\ell}$ in the construction of the integral period morphism.

\subsection{Isogeny Classes}\label{expisog} Consider the limits $\shS_{\fK_p}(G, \Ohm) = \varprojlim_{\fK^p} \shS_{\fK_p \fK^p}(G, \Ohm)$ and $\Sh_{\fK_p}(G, \Ohm) = \varprojlim_{\fK^p} \Sh_{\fK_p \fK^p}(G, \Ohm)$ where $\fK^p$ varies over compact open subgroups of $G(\IA_f^p)$. Similarly, take the limit $\shS_{\sK_p}(\GSp, \sH)$. We explain the notion of isogeny classes of points on $\shS_{\fK_p}(G, \Ohm)$ under the assumption that $\psi$ induces a self-dual pairing on $H_{\IZ_{(p)}}$\footnote{This assumption is imposed in \cite[(1.3.3)]{Modp}. We do not want to replace $H$ by $H^{\tensor (4, 4)}$.}. 

\begin{definition}
\label{CSpin-isog} 
Let $k$ be a perfect field and $s, s' \in \shS_{\fK}(G, \Ohm)(k)$. Let $f : \sA_s \to \sA_{s'}$ be a quasi-isogeny over $k$ which respects the $\IZ/2\IZ$-grading and $\Cl(L)$-action on $\sA_s, \sA_{s'}$. Let $\bar{k}$ be an aglebraic closure of $k$. If $\mathrm{char\,} k = 0$, we say $f$ is a \textit{CSpin-isogeny} if it sends $\bpi_{\ell, s \tensor \bar{k}}$ to $\bpi_{\ell, s' \tensor \bar{k}}$ for every prime $\ell$. If $\mathrm{char\,} k = p$, we say $f$ is a \textit{CSpin-isogeny} if it sends $\bpi_{\ell, s \tensor \bar{k}}$ to $\bpi_{\ell, s' \tensor \bar{k}}$ for every prime $\ell \neq p$ and $\bpi_{\cris, s}$ to $\bpi_{\cris, s'}$. 
\end{definition}

\subsubsection{}\label{3.2.1} We now begin to explain the notion of isogeny classes for $\bar{\IF}_p$ points, following \cite{Modp}. We write $W$ for $W(\bar{\IF}_p)$ and $K_0$ for $W[1/p]$. Let $s$ be an $\bar{\IF}_p$-point on $\shS_{\fK_p}(G, \Ohm)$. There is a triple $(\sA_s, \mu_s, \epsilon^p_s)$ attached to $s$, where $(\sA_s, \mu_s)$ is a polarized abelian variety over $\bar{\IF}_p$ and $\epsilon^p_s$ is an isomorphism $H \tensor \what{\IZ}^p \to H^1_\et(\sA_s, \what{\IZ}^p)$ which respects the pairings induced by $\psi$ and $\mu_s$ up to a $(\IA_f^p)^\times$-multiple. Moreover, $\epsilon^p_s$ respects the $\IZ/2\IZ$-grading and $\Cl(L)$-action and sends $\pi$ to $\bpi_{\IA^p_f, s}$ (cf. \cite[(1.3.6)]{Modp}). 

By the extension property of the integral model $\shS_{\fK_p}(G, \Ohm)$, the (right) $G(\IA_f^p)$-action on $\Sh_{\fK_p}(G, \Ohm)$ extends to $\shS_{\fK_p}(G, \Ohm)$. Take $g^p \in G(\IA_f^p)$. If $s' = s \cdot g^p$, then the triple $(\sA_{s'}, \mu_{s'}, \epsilon_{s'}^p)$ is isomorphic to the triple $(\sA_s, \mu_s, \epsilon_s^p \circ g^p)$ when we view $(\sA_{s'}, \mu_{s'})$ and $(\sA_s, \mu_s)$ as weakly polarized abelian schemes up to prime-to-$p$ isogeny. In other words, there exists a $p'$-quasi-isogeny $f : \sA_s \to \sA_{s'}$ which sends $\mu_{s'}$ to a $(\IA^p_f)^\times$-multiple of $\mu_s$ and the composition 
$$ H \tensor \IA^p_f \stackrel{\epsilon_{s'}^p \tensor \IQ}{\to} H^1_\et(\sA_{s'}, \IA_f^p) \stackrel{f^*}{\to} H^1_\et(\sA_s, \IA_f^p) $$
is equal to $(\epsilon_s^p \tensor \IQ) \circ g^p$. Such $f$ is unique, as the induced map $H^1_\et(\sA_{s'}, \IA^p_f) \to H^1_\et(\sA_s, \IA^p_f)$ is completely determined. We argue that $f$ is a CSpin-isogeny. We only need to check that $f$ sends $\bpi_{\cris, s}$ to $\bpi_{\cris, s'}$. This is easily done by a lifting argument: Lift $s$ to a $W$-point $s_W$ on $\shS_{\fK_p}(G, \Ohm)$ and set $s'_W := s_W \cdot g^p$. Then $s'_W$ is a $W$-point which lifts $s'$. Moreover, there is a $p'$-quasi-isogeny $f_W : \sA_{s_W} \to \sA_{s_W'}$ which lifts $f$. Let $s_{K_0}$ and $s_{K_0}'$ be the generic points of $s_W$ and $s_W'$. Choose an algebraic closure $\bar{K}_0$ of $K_0$ and let $\bar{s}_{K_0}, \bar{s}_{K_0}'$ be the corresponding geometric points over $s_{K_0}, s_{K_0}'$. We argue that $f_W \tensor K_0 : \sA_{s_{K_0}} \to \sA_{s_{K_0}'}$ sends $\bpi_{p, \bar{s}_{K_0}}$ to $\bpi_{p, \bar{s}'_{K_0}}$. Choose an isomorphism $\bar{K}_0 \iso \IC$ and let $s_\IC, s_\IC'$ be the $\IC$-points given by $\bar{s}_{K_0}, \bar{s}_{K_0}'$. It suffices to check that $f_W \tensor \IC : \sA_{s_\IC} \to \sA_{s_\IC'}$ sends $\bpi_{B, s}$ to $\bpi_{B, s'}$, but this is clear: By the smooth and proper base change theorem, $f_W \tensor \IC$ sends $\bpi_{\IA_f^p, s_\IC}$ to $\bpi_{\IA^p_f, s_\IC'}$ as $f$ sends $\bpi_{\IA_f^p, s}$ to $\bpi_{\IA_f^p, s'}$. One deduces that $f$ sends $\bpi_{\cris, s}$ to $\bpi_{\cris, s'}$ using that $f_W \tensor K_0$ sends $\bpi_{p, \bar{s}_{K_0}}$ to $\bpi_{p, \bar{s}'_{K_0}}$ (cf. Rmk~\ref{dRham}). 

\subsubsection{}\label{pisog} We fix an isomorphism of $\IZ/2\IZ$-graded right $\Cl(L)$-modules $H^1_\cris(\sA_s/W) \iso H \tensor W$ which sends $\bpi_{\cris, s}$ to $\pi$ (cf. \cite[Prop.~4.7]{CSpin}).  Then the Frobenius action on $H^1_\cris(\sA_s /W)$ takes the form $b\sigma$ for some $b \in G(W)$. By \cite[Lem.~1.1.12]{Modp}, there exists a $G_W$-valued cocharacter $\nu$ such that  
$b \in G(W) \nu(p) G(W)$ and $\sigma^{-1}(\nu)$ gives the filtration on $H^1_\dR(\sA_s/\bar{\IF}_p)$. Now define $$X_p := \{ g_p \in G(K_0)/G(W) : g_p^{-1} b \sigma(g_p) \in G(W) \nu(p) G(W) \}. $$
Let $\shG_s$ be the $p$-divisible group of $\sA_s$. If $g_p \in X_p$, then $g_p \cdot \ID(\shG_s)$ is stable under Frobenius and satisfies the axioms of a Dieudonn\'e module. Hence $g_p \cdot \ID(\shG_s)$ corresponds to a $p$-divisible group $\shG_{g_p s}$ naturally equipped with a quasi-isogeny $\shG_s \to \shG_{g_p s}$. We denote by $\sA_{g_ps}$ the corresponding abelian variety, such that the $p$-divisible group of $\sA_{g_p s}$ is identified with $\shG_{g_p s}$ and there is a quasi-isogeny $f_{g_p} : \sA_{s} \to \sA_{g_p s}$ which induces $\shG_s \to \shG_{g_p s}$. We can view $\sA_{g_p s}$ as an object in $\shA(\bar{\IF}_p) \tensor \IZ_{(p)}$ and equip it with a weak polarization and a level structure using those on $\sA_s$ via $f_{g_p}$. This gives rise to a point in $\shS_{\sK_p}(\GSp, \sH)(\bar{\IF}_p)$, which we denote by $g_p s$. Thus we obtain a map $X_p \to \shS_{\sK_p}(\GSp, \sH)(\bar{\IF}_p)$ (cf. Sect.~\ref{GSpMod}) by sending $g_p$ to $g_p s$. By \cite[Prop.~1.4.4]{Modp}, this map lifts uniquely to a map $\iota_{p} : X_p \to \shS_{\fK_p}(G, \Ohm)(\bar{\IF}_p)$ such that $f_{g_p} : \sA_s \to \sA_{\iota_{p}(g_p)} = \sA_{g_p s}$ respects the $\IZ/2\IZ$-grading, $\Cl(L)$-action and sends $\bpi_{\cris, s}$ to $\bpi_{\cris, \iota_{p}(g_p)}$. Since the level structure on $\sA_{\iota_{p}(g_p)}$ is induced from the one on $\sA_{s}$ via $f_{g_p}$, $f_{g_p}$ sends $\bpi_{\IA^p_f, s}$ to $\bpi_{\IA^p_f, \iota_{p}(g_p)}$. Therefore, $f_{g_p s}$ is a CSpin-isogeny.

Set $X^p = G(\IA_f^p)$. We construct a map $\iota_s : X_p \times X^p \to \shS_{\fK_p}(G, \Ohm)(\bar{\IF}_p)$ by sending the pair $(g_p, g^p)$ to $[\iota_{p}(g_p)] \cdot g^p$. The image of $\iota_s$ is called the \textit{isogeny class} of $s$. By \cite[Prop.~1.4.15]{Modp} and its proof, we have: 

\begin{proposition}
\label{1.4.15}
Let $s \in \shS_{\fK_p}(G, \Ohm)(\bar{\IF}_p)$ be a point. If $s' = [\iota_{p}(g_p)] \cdot g^p$ for $(g_p, g^p) \in X_p \times X^p$, then there exists a CSpin-isogeny $f : \sA_s \to \sA_{s'}$ such that $f^* H^1_\cris(\sA_{s'}/W) = g_p \cdot H^1_\cris(\sA_{s}/W)$ and $f^* H^1_\et(\sA_{s'}, \what{\IZ}^p) = g^p \cdot H^1_\et(\sA_s, \what{\IZ}^p)$.  
\end{proposition}

\subsubsection{} We now study the images on $\shS_{\fK^\ad_p}(G^\ad, \Ohm)(\bar{\IF}_p)$ of isogeny classes on $\shS_{\fK_p}(G, \Ohm)(\bar{\IF}_p)$. Let $s \in \shS_{\fK_p}(G, \Ohm)(\bar{\IF}_p)$ and let $X_p, X^p$ be as above. Let $X^{\ad}_p$ and $X^{\ad, p}$ be the quotients of $X_p$ and $X^p$ under the action of $\IG_m(\IQ_p)$ and $\IG_m(\IA^p_f)$ respectively. 
\begin{lemma}
$\iota_s$ descends to a map $\iota^\ad_s$ which fits into a commutative diagram 
\begin{center}
\begin{tikzcd}
X_p \times X^p \arrow{r}{\iota_s} \arrow{d}{} & \shS_{\fK_p}(G, \Ohm)(\bar{\IF}_p) \arrow{d}{} \\
X^\ad_p \times X^{\ad, p} \arrow{r}{\iota^\ad_s} & \shS_{\fK^\ad_p}(G^\ad, \Ohm)(\bar{\IF}_p)
\end{tikzcd}
\end{center}
\end{lemma} 
\begin{proof}
$\iota_s$ is equivariant under the action of  $\IG_m(\IQ_p) \times \IG_m(\IA^p_f) \subset Z_G(\IQ_p) \times G(\IA_f^p)$ (cf. \cite[Cor.~1.4.13]{Modp} ). By (\ref{Delta}), $\shS_{\fK_p}(G, \Ohm) \to \shS_{\fK^\ad_p}(G^\ad, \Ohm)$ is a pro-\'etale cover which is a torsor of the group $\Delta = \IG_m(\IA_f^p)/ \IG_m(\IZ_{(p)})$. One may easily check, for example by examining the generic fiber, that the action of $\IG_m(\IQ_p) \times \IG_m(\IA^p_f)$ factors through $\Delta$ via the quotient map 
$$ \IG_m(\IQ_p) \times \IG_m(\IA^p_f) \to [\IG_m(\IQ_p)/\IG_m(\IZ_p) \times \IG_m(\IA^p_f)]/\IG_m(\IQ) \stackrel{\sim}{\to} \IG_m(\IA_f^p)/ \IG_m(\IZ_{(p)}) = \Delta $$

\end{proof}

The isomorphism $H^1_\cris(\sA_s/W) \iso H \tensor W$ we fixed in section~\ref{pisog} gives us an isometry $\bL_{\cris, s} \iso L \tensor W$. Set $\nu^\ad := \ad \circ \nu$ and $b^\ad := \ad \circ b$.  
\begin{lemma}
\label{surjatp}
 $X^\ad_p = \{ g_p' \in G^\ad(K_0)/G^\ad(W) :  (g_p')^{-1} b^\ad \sigma(g_p') \in G^\ad(W) {\nu^\ad}(p) G^\ad(W) \} $
\end{lemma}
\begin{proof}
Let $g_p' \in G^\ad(K_0)$ be any representative of an element of $X_p^\ad$ and let $g_p \in G(K_0)$ be any lift of $g_p'$. We want to show that $g^{-1}_p b \sigma(g_p) \in G(W) \nu(p) G(W)$. 

Set $\wt{G} := \IG_m \times G^\ad$. Consider the central isogeny 
$$ \Nm \times \ad : G \to \wt{G} $$
Let $T \subset G_{\IZ_p}$ (resp. $\wt{T} \subset \wt{G}_{\IZ_p}$) be the centralizer of a maximal split torus and let $\Ohm_{G}$ (resp. $\Ohm_{\wt{G}}$) be the associated Weyl group. We can arrange that $T = \wt{T} \times_{\wt{G}_{\IZ_p}} G_{\IZ_p}$ and $\nu \in X_*(T)$. The Cartan decomposition gives us an isomorphism 
$$ X_*(T) / \Ohm_G \stackrel{\sim}{\to} G(W) \backslash G(K) / G(W) \text{ defined by } \mu \mapsto \mu(p)$$
and an analogous one for $\wt{T} \subset \wt{G}_{\IZ_p}$.   

The element $g^{-1}_p b \sigma(g_p) \in G(W) \nu'(p) G(W)$ for some cocharacter $\nu' \in X_*(T)$. As the central isogeny $\Nm \times \ad$ induces an injection $X_*(T)/\Ohm_G \to X_*(\wt{T})/\Ohm_{\wt{G}}$, it suffices to check that $(\Nm \times \ad) \circ \nu$ and $(\Nm \times \ad) \circ \nu'$ lie in one $\Ohm_{\wt{G}}$-orbit. By assumption, $G^\ad(W) (\ad \circ \nu') (p) G^\ad(W) = G^\ad(W) \nu^\ad(p) G^\ad(W)$, so it remains to check that $$\IG_m(W) (\Nm \circ \nu')(p) \IG_m(W) = \IG_m(W) (\Nm \circ \nu)(p) \IG_m(W).$$ 
This follows easily from the observation that $\val_p(\Nm(g_p^{-1} b \sigma(g_p))) = \val_p(\Nm(b))$ for any $p$-adic valuation $\val_p$.
\end{proof}
\begin{lemma}
\label{surjprimetop}
$ X^{\ad, p}= G^\ad(\IA_f^p)$
\end{lemma}
\begin{proof}
By Cartan-Dieudonn\'e theory, $G(\IQ_\ell) / \IG_m(\IQ_\ell) \to G^\ad(\IQ_\ell)$ is surjective for every $\ell \neq p$, so we already know that $\prod_{\ell \neq p} G(\IQ_\ell) \to \prod_{\ell \neq p} G^\ad(\IQ_\ell)$ is surjective. We have embeddings $G \into \GL(H_{\IZ_{(p)}})$ (resp. $G^\ad \into \GL(L_{\IZ_{(p)}})$ given by the spin representation (resp. standard representation). An element $g = (g_\ell)_{\ell \neq p} \in \prod_{\ell \neq p} G(\IQ_\ell)$ (resp. $g' = (g'_\ell)_{\ell \neq p} \in \prod_{\ell \neq p} G^\ad(\IQ_\ell)$) lies in $G(\IA_f^p)$ (resp. $G^\ad(\IA_f^p)$) if and only if for all but finitely many $\ell \neq p$, $g_\ell \in \GL(H \tensor \IZ_\ell)$ (resp. $g'_\ell \in \GL(L \tensor \IZ_\ell)$). 

It suffices to check that for all but finitely many $\ell \neq p$, if $g'_\ell \in G^\ad(\IQ_\ell) \cap \GL(L \tensor \IZ_\ell)$, $g'_\ell$ can be lifted to an element $g_\ell$ in $G(\IQ_\ell) \cap \GL(H \tensor \IZ_\ell)$. We claim this is true for every odd $\ell \neq p$ such that $L \tensor \IZ_\ell$ is self dual. In this case, the $\IQ$-groups $G_\IQ$ and $G^\ad_\IQ$ have natural models $G_{\IZ_{(\ell)}}$ and $G^\ad_{\IZ_{(\ell)}}$ over $\IZ_{(\ell)}$ such that $G_{\IZ_{(\ell)}}(\IZ_\ell) = G(\IQ_\ell) \cap \GL(H \tensor \IZ_\ell)$ and $G^\ad_{\IZ_{(\ell)}}(\IZ_\ell) = G^\ad(\IQ_\ell) \cap \GL(L \tensor \IZ_\ell)$. Now by Cartan-Dieudonn\'e theory again, the map $G_{\IZ_{(\ell)}}(\IF_\ell) \to G^\ad_{\IZ_{(\ell)}}(\IF_\ell)$ is surjective. The map $G_{\IZ_{(\ell)}} \to G^\ad_{\IZ_{(\ell)}}$ is smooth because the kernel $\IG_{m, \IZ_{(\ell)}}$ is smooth. Therefore, $G_{\IZ_{(\ell)}}(\IZ_\ell) \to G^\ad_{\IZ_{(\ell)}}(\IZ_\ell)$ is also surjective. 
\end{proof}

\subsection{The Period Morphism}
\label{period}
\subsubsection{} Let $\IU$ denote the standard hyperbolic plane, i.e., the quadratic lattice over $\IZ$ of rank 2 with a basis $x, y$ such that $x^2 = y^2 = 0$ and $\< x, y \> = 1$. Let $\IE_8$ be the unique unimodular positive definite even lattice of rank $8$. As is customary in the study of K3 surfaces, we equip the second cohomology of K3 surfaces with the \textit{negative} Poincar\'e pairing, so that the K3 lattice $\Lambda$ is isomorphic to $\IU^{\oplus 3} \oplus \IE^8_2$. Fix a prime $p> 2$ and $d \in \IZ_{>0}$ which is prime to $p$. Let $e, f$ be a basis of the first copy of the hyperbolic plane $\IU$ such that $e^2 = f^2 = 0, \<e, f \> = 1$ and set $L_d = \< e - df \>^\perp$.  $L_d$ is an integral lattice of discriminant $2d$ and is abstractly isomorphic to $P^2(X_\IC, \IZ)$ for any primitively quasi-polarized K3 surface $(X_\IC, \xi_\IC)$ of degree $\xi_\IC^2 = 2d$ over $\IC$. Recall that a quasi-polarization is said to be primitive if it is not a positive power of another line bundle. 

\subsubsection{} A K3 surface over a scheme $S$ is a proper and smooth algebraic space $f : X \to S$  whose geometric fibers are K3 surfaces. A \textit{polarization} (resp. \textit{quasi-polarization}) of a K3 surface $X \to S$ is a section $\xi \in \underline{\Pic}(X/S)(S)$ whose fiber at each geometric point is an ample (resp. big and nef) line bundle. A section $\xi$ is called \textit{primitive} if for all geometric points $s \to S$, $\xi(s)$ is primitive. $(X \to S, \xi)$ is said to be of degree $2d$ if for all geometric points $s \to S$, $\xi(s)$ has degree $2d$. 

Let $\Mod_{2d}$ (resp. $\Mod_{2d}^\circ$) be the moduli problem over $\IZ[1/2]$ which sends each $\IZ[1/2]$-scheme $S$ to the groupoid of tuples $(f: X \to S, \xi)$, where $f : X \to S$ is a K3 surface and $\xi$ is a primitive quasi-polarization (resp. polarization) of $X$ with $\deg(\xi) = 2d$. $\Mod^\circ_{2d}$ and $\Mod_{2d}$ are Deligne-Mumford stacks of finite type over $\IZ$, $\Mod^\circ_{2d}$ is separated, and the natural map $\Mod^\circ_{2d} \to \Mod_{2d}$ is an open immersion (cf. \cite{Rizov1}, \cite[2.1]{Maulik}, \cite[3.1]{Keerthi}).

Let $(\sX \to \Mod_{2d}, \bxi)$ be the universal object over $\Mod_{2d}$. For each scheme $T \to \Mod_{2d}$, we denote by $(\sX_T, \bxi_T)$ the pullback of $(\sX, \boldsymbol{\xi})$. Let $\bH_B^2$ be the second relative Betti cohomology of $\sX_\IC \to \Mod_{2d, \IC}$. For every $\ell$, let $\bH^2_\ell$ be the relative second \'etale cohomology of the restriction of $\sX$ to $\Mod_{2d, \IZ[1/2\ell]}$ with coefficients in $\underline{\IZ}_\ell$. Let $\bH_{\dR}^2$ be the vector bundle on $\Mod_{2d}$ given by the second relative de Rham cohomology of $\sX$, which comes with a natural filtration $\Fil^\bullet$. For $* = B, \ell, \dR$, denote by $\bP^2_*$ the primitive part of $\bH^2_*$, i.e., the orthogonal complement to the Chern class of $\bxi$. We put together the relative $\ell$-adic cohomology sheaves $\bH_\ell^2$ to form $\bH_{\what{\IZ}^p}^2 := \prod_{\ell \neq p} \bH_\ell^2$ and similarly we put together Chern classes to form $\ch_{\what{\IZ}^p}(\bxi)$ in $\bH^2_{\what{\IZ}^p}$.

For each scheme $T \to \Mod_{2d, \IF_p}$, we can additionally consider the crystal of vector bundles $\bH^2_{\cris, T}$ on $\Cris(T/\IZ_p)$ given by the second crystalline cohomology of $\sX_T$. Let $\bP^2_{\cris, T}$ be the primitive part of $\bH_{\cris, T}$. 

\subsubsection{} We use the notations of Sect.~\ref{Shimura} by setting $L = L_d$. Let $\psi : H \times H \to \IZ$ be the symplectic pairing constructed in \cite[Sect.~5.1]{Maulik}, so that $\psi$ induces a perfect pairing on $H \tensor \IZ_{(p)}$. Again let $\delta \in \det(L)$ be some fixed nonzero element and $m$ be the rank of $L$. 

Let $\wt{\Mod}_{2d} \to \Mod_{2d}$ be the \'etale $2$-cover such that for every scheme $T$, a point in $\wt{\Mod}_{2d}(T)$ is given by a pair $(T \to \Mod_{2d}, \bdeta_{2, T})$, where $\bdeta_{2, T}$ is a section in $H^0(T, \det(\bP_{2, T}))$ such that $\< \bdeta_{2, T}, \bdeta_{2, T} \>$ is the constant section $\< \delta, \delta \>$. We call it the \textit{orientation cover.} The universal property of the Shimura stack $\Sh_{\fK^\ad_L}(G^\ad, \Ohm)_\IC$ induces a map $\rho_\IC : \wt{\Mod}_{2d, \IC} \to \Sh_{\fK^\ad_L}(G^\ad, \Ohm)_\IC$ (cf. \cite[Prop.~4.3, 5.2]{Keerthi}).  $\rho_\IC$ descends to a map $\rho_\IQ : \wt{\Mod}_{2d, \IQ} \to \Sh_{\fK^\ad_L}(G^\ad, \Ohm)$ defined over $\IQ$ (cf. \cite[Cor.~5.4]{Keerthi}, see also \cite[3.16]{Rizov}). By construction, there is an isomorphism $\alpha_B : \rho^*_\IC \bL_B \to \bP^2_B(1)$. Deligne's big monodromy argument in \cite[6.4]{Deligne1} can be used to show that for each rational prime $\ell$, there exists a (necessarily unique) isomorphism $\alpha_{\ell, \IQ} : \rho^* \bL_{\ell, \IQ} \to \bP^2_{\ell, \IQ}$ of \'etale local systems over $\wt{\Mod}_{2d, \IQ}$ (cf. \cite[Prop.~5.6]{Keerthi}.) which is compatible with $\alpha_B$ via Artin's comparison isomorphism. Let $\bdeta_{\ell, \IQ}$ be the pullback of the global section $\bd_{\ell, \IQ}$ of $\bL_{\ell, \IQ}$. Since $\wt{\Mod}_{2d}$ is normal (cf. \cite[Cor.~3.9]{Keerthi}), we can extend $\bdeta_{\ell, \IQ}$ to a global section $\bdeta_{\ell}$ of $\det(\bP_\ell^2)$ over $\wt{\Mod}_{2d, \IZ[(2\ell)^{-1}]}$ for every $\ell$.

For our purposes it suffices to look at $\wt{\Mod}_{2d, \IZ_{(p)}}$. Let $I^p$ be the \'etale sheaf over $\Mod_{2d, \IZ_{(p)}}$ which associates to each morphism $S \to \Mod_{2d, \IZ_{(p)}}$ the set 
$$ \{ \text{isometries } \eta : \Lambda \tensor \underline{\what{\IZ}}^p \stackrel{\sim}{\to} \bH^2_{\what{\IZ}^p, S} \text{ such that } \eta(e - df) = \mathrm{ch}_{\what{\IZ}^p}(\bxi_S), (\wedge^m \eta)(\delta \tensor 1) = (\bdeta_{\ell, S})\}. $$ It comes equipped with a natural action by the constant sheaf of $\fK_L^{p,\ad}$. For every compact open subgroup $\fK' \subset \fK_L^{\ad}$ of the form $\fK^\ad_p K'^p$ for some $\fK'^p \subset G^\ad(\IA_f^p)$, we let $\wt{\Mod}_{2d, \fK', \IZ_{(p)}}$ denote the relative moduli problem over $\wt{\Mod}_{2d, \IZ_{(p)}}$ which attaches to each morphism $S \to \wt{\Mod}_{2d, \IZ_{(p)}}$ the set $H^0(S, I^p/\fK'^p)$. A section $[\eta] \in H^0(S, I^p/{\fK'}^p)$ is called a $\fK'$-level structure on $S$. Denote by $\wt{\Mod}_{2d, \fK', \IZ_{(p)}}$ the corresponding pullback of $\wt{\Mod}_{2d, \IZ_{(p)}}$. Denote by $\wt{\Mod}_{2d, \fK^\ad_p, \IZ_{(p)}}$ the limit $\varprojlim_{\fK'^p} \wt{\Mod}_{2d, \fK^\ad_p K'^p, \IZ_{(p)}}$ as $\fK'^p$ varies over the compact open subgroups of $G^\ad(\IA^p_f)$. 

\begin{remark}
There is a minor issue in Rizov's construction of the period map in \cite[3.9]{Rizov}, see \cite[Rmk~5.12]{Taelman2}.\footnote{The author is grateful to an anonymous referee for point out the issue with orientations to him.} Our definition of level structures is a slight generalization of Def.~5.11 in \cite{Taelman2}.
\end{remark}

\subsubsection{} Let $\fK' \subset \fK_L^\ad$ be a compact open subgroup of the form $\fK^\ad_p \fK'^p$ with $\fK'^p \subset G^\ad(\IA^p_f)$ sufficiently small. The map $\rho_{\IQ}$ lifts naturally to a map $\rho_{\fK', \IQ} : \wt{\Mod}_{2d, \fK', \IQ} \to \Sh_{\fK'}(G^\ad, \Ohm)$, and the extension property of the canonical integral model allows us then to extend the map $\rho_{\fK', \IQ}$ to (cf. \cite[Prop.~5.7]{Keerthi})
$$ \rho_{\fK', \IZ_{(p)}} :   \wt{\Mod}_{2d, \fK', \IZ_{(p)}} \to \shS_{\fK'}(G^\ad, \Ohm). $$
In turn, by taking quotients, we obtain a morphism of Deligne-Mumford stacks $\rho_{\IZ_{(p)}} : \wt{\Mod}_{2d, \IZ_{(p)}} \to \shS_{\fK^\ad_L}(G^\ad, \Ohm)$. It follows from local Torelli theorem for K3 surfaces that $\rho_\IC$, and hence $\rho_\IQ$, is \'etale. Madapusi-Pera showed that $\rho_{\IZ_{(p)}}$ is still \'etale \cite[Thm~5.8]{Keerthi}.

A major part of \cite{Keerthi} is to show that $\rho_{\IZ_{(p)}}$ still records the cohomological data of K3 surfaces, just like $\rho_{\IC}$. This is done by comparing sheaves on $\wt{\Mod}_{2d, \IZ_{(p)}}$ with those on $\shS_{\fK_L^\ad}(G^\ad, \Ohm) $ through $\rho_{ \IZ_{(p)}}$. For simplicity, we write $\rho_{\IZ_{(p)}}$ simply as $\rho$. By construction, we have isometries $\alpha_{B} : \rho^* \bL_B \stackrel{\sim}{\to} \bP^2_B(1)$ and $\alpha_{\dR} : \rho^* \bL_{\dR} \stackrel{\sim}{\to} \bP^2_\dR(1)$ over $\wt{\Mod}_{2d, \IC}$. We summarize some properties which we shall need in the following proposition:

\begin{proposition}
\label{compareCoh}
\begin{enumerate}[label=(\alph*)]
\item $\alpha_B$ extends to an isometry of $\IZ_\ell$ \'etale local systems $\alpha_\ell : \rho^* \bL_\ell(-1) \stackrel{\sim}{\to} \bP_\ell^2$ over $\wt{\Mod}_{2d, \IZ_{(p)}}$ for every $\ell \neq p$ and an isometry of $\IZ_p$ \'etale local systems $\alpha_p : \rho^* \bL_p(-1) \stackrel{\sim}{\to}  \bP_p^2$ over $\wt{\Mod}_{2d,\IQ}$. 
\item $\alpha_{\dR, \IC}$ extends to an isometry of filtered vector bundles with flat connection $\alpha_{\dR} : \rho^* \bL_{\dR}(-1) \stackrel{\sim}{\to} \bP^2_\dR$ which sends $\Fil^2 \bL_{\dR} = \Fil^1 \bL_\dR(-1)$ to $\Fil^1 \bP^2_\dR$ over  $\wt{\Mod}_{2d, \IZ_{(p)}}$. 
\item Let $T \to \wt{\Mod}_{2d, \IF_p}$ be an \'etale map. $\alpha_\dR$ induces a canonical isomorphism $\alpha_{\cris, T} : \rho^* \bL_{\cris}(-1)_T \stackrel{\sim}{\to} \bP^2_{\cris, T}$, where $\rho^* \bL_{\cris}(-1)_T$ is the crystal on $\Cris(T/\IZ_p)$ defined by pulling back $\bL_\cris(-1)$ via the composite morphism $T \to \wt{\Mod}_{2d, \IF_p} \stackrel{\rho}{\to} \shS_{\fK^\ad_L}(G^\ad, \Ohm)_{\IF_p}$. $\alpha_{\cris, T}$ is a morphism of F-crystals. 
\end{enumerate}
\end{proposition}
\begin{proof}
These results are proved in \cite[Sect.~5]{Keerthi}. 
\end{proof}

Technically, the universal family over $\wt{\Mod}_{2d}$ is \'etale locally an algebraic space, whereas the references for de Rham or crystalline comparison isomorphisms used in \cite[Sect.~5]{Keerthi} often only consider schemes. This issue is addressed in \cite[Sect.~11]{IIK}. Here we give a brief sketch: Let $k$ be a perfect field of characteristic $p$, $V$ be the ring of integers of a finite totally ramified extension $K$ of $W(k)[1/p]$, and $Y_V$ be an algebraic space over $V$. Assume that the generic fiber $Y_K$ and the special fiber $Y_k$ are both schemes. This assumption is always satisfied, for example, if $Y_V$ is a K3 space, because smooth proper algebraic spaces over a field of dimension $\le 2$ are schemes. It has been explained in \cite[Sect.~11.2]{IIK} that under this assumption, \cite[Thm~14.6]{BMS} and GAGA results for \'etale cohomology such as \cite[Thm~3.7.2]{GAGA} give rise to an isomorphism $H^*_\cris(Y_k/W(k)) \tensor_{W(k)} B_\cris \stackrel{\sim}{\to} H^*_\et(Y_K \tensor_K \bar{K}, \IZ_p) \tensor_{\IZ_p} B_\cris$. Moreover, it is compatible via the Berthelot-Ogus isomorphism with the isomorphism $H^*_\dR(Y_K/K) \tensor B_\dR \stackrel{\sim}{\to} H^*_\et(Y_K \tensor \bar{K}, \IZ_p) \tensor_{\IZ_p} B_\dR$ provided by \cite[Thm~8.4]{Scholze} and GAGA results. \cite[Sect.~11.3]{IIK} explained that Blasius' results in \cite{Blasius} hold with these comparison isomorphisms and the $X$ in \cite[Thm~3.1]{Blasius} is allowed to be an algebraic space. These observations fill in the details in the proofs of \cite[Prop.~5.3, 5.6(4)]{Keerthi}.

\subsection{Special Endomorphisms} Let $\fK^p \subset \fK_L^p$ be a sufficiently small compact open subgroup and set $\fK = \fK_p \fK^p \subset G(\IA_f)$. We write $\rho_{\fK^\ad, \IZ_{(p)}} : \wt{\Mod}_{2d, \fK^\ad, \IZ_{(p)}} \to \shS_{\fK^\ad}(G^\ad, \Ohm)$ simply as $\rho$.

\begin{definition}
\emph{(cf. \cite[5.22]{CSpin})}
Let $k$ be a perfect field with algebraic closure $\bar{k}$, $s \in \shS_{\fK}(G, \Ohm)(k)$ and $f \in \End(\sA_s)$ defined over $k$.\footnote{$\shS_{\fK}(G, \Ohm) := \shS_{\fK_p}(G, \Ohm)/\fK^p$ is a Deligne-Mumford stack which carries a descent of the universal abelian scheme $\sA$ on $\shS_{\fK_p}(G, \Ohm)$.} If $\mathrm{char\,} k = 0$, we say $f$ is a \textit{special endomorphism} if its $\ell$-adic realization lies in $\bL_{\ell, s \tensor \bar{k}} \tensor \IQ_\ell \subset \End (\bH_{\ell, s \tensor \bar{k}} \tensor \IQ_\ell)$ for every prime $\ell$. If $\mathrm{char\,} k = p$, we say $f$ is a \textit{special endomorphism} if its $\ell$-adic realization lies in $\bL_{\ell, s \tensor \bar{k}} \tensor \IQ_\ell  \subset \End (\bH_{\ell, s \tensor \bar{k}} \tensor \IQ_\ell)$ and its crystalline realization lies in $\bL_{\cris, s}[1/p] \subset \End (\bH_{\cris, s}[1/p])$.  
\end{definition} In either case, the special endomorphisms form a subspace $L(\sA_s) \subset \End(\sA_s)$. If $f \in L(\sA_s)$, then $f \circ f$ is a scalar. Therefore, $L(\sA_s)$ has the structure of a quadratic lattice given by $f \mapsto f \circ f$. Now we relate special endomorphisms to line bundles on K3 surfaces. 
\begin{proposition}
\label{submotive}
Let $t \in \wt{\Mod}_{2d, \fK^\ad, \IZ_{(p)}}(\bar{\IF}_p)$ be a point and $(\sX_t, \bxi_t)$ be the associated quasi-polarized K3 surface. Let $s \in \shS_{\fK}(G, \Ohm)(\bar{\IF}_p)$ be a lift of $\rho(t)$. Let $ \< \bxi_t \>_\IQ^\perp$ denote the orthogonal complement of $\bxi_t$ in $\Pic(\sX_t)_\IQ$. Then there is an isomorphism of quadratic lattices over $\IQ$
$$ L(\sA_s)_\IQ \iso \< \bxi_t \>_\IQ^\perp $$
whose $\ell$-adic and crystalline realizations agree with the isomorphisms $\bL_{\ell, s} = \bL_{\ell, \rho(t)} \stackrel{\sim}{\to} P^2_\et(\sX_t, \IZ_\ell)$ and $\bL_{\cris, s} = \bL_{\cris, \rho(t)} \stackrel{\sim}{\to} P^2_\cris(X/W)$ given by $\alpha_\ell$ and $\alpha_\cris$ in Proposition~\ref{compareCoh}.
\end{proposition}
Note that we have identifications $\bL_{\ell, s} = \bL_{\ell, \rho(t)}$ and $\bL_{\cris, s} = \bL_{\cris, \rho(t)}$ because $\bL_\ell$ and $\bL_\cris$ on (the appropriate fibers of) $\shS_{\fK^\ad}(G^\ad, \Ohm)$ are descent of the corresponding sheaves on $\shS_{\fK}(G, \Ohm)$. 
\begin{proof}
This is a coarser form of \cite[Thm~5.17(4)]{Keerthi}, except that we allow quasi-polarized K3 surfaces as well. We slightly adapt the arguments in \textit{loc. cit.} Let $\alpha_{\cris, t}$ be the isomorphism provided by Prop.~\ref{compareCoh}(c) for $T=t$. It suffices to construct a map $L(\sA_s) \to \< \bxi_t \>^\perp \subset \Pic(\sX_t)$ such that the following diagram commutes,
\begin{center}
    \begin{tikzcd}
    L(\sA_s) \arrow{r}{} \arrow{d}{} & \< \bxi_t \>^\perp  \arrow{d}{} \\
    \bL_{\cris, s}(-1)^{F = p} \arrow{r}{\alpha_{\cris, t}} & (P^2_\cris(\sX_t/W))^{F = p}
    \end{tikzcd}
\end{center}
because the left vertical arrow is known to be an isomorphism after the domain is tensored with $\IQ_p$ by \cite[Theorem 6.4(2)]{Keerthi}.

Let $f \in L(\sA_s)$ be a special endomorphism. By \cite[Prop.~5.21]{CSpin}, there exists a characteristic $0$ field $F$ and an $F$-point $s_F$ on $\shS_{\fK}(G, \Ohm)$ which specializes to $s$ such that $f$ lifts an element $f_F \in L(\sA_{s_F})$. By the \'etaleness of the period map, there exists an $F$-point $t_F$ on $\wt{\Mod}_{2d, \fK^\ad, \IZ_{(p)}}$ which specializes to $t$ and whose image in $\shS_{\fK^\ad}(G^\ad, \Ohm)$ lifts to $s_F$. Choose an embedding $F \into \IC$ and let $t_\IC, s_\IC$ denote the base change of $t_F, s_F$ to $\IC$. Since we may identify $\bL_{B, s_{\IC}}$ with $P^2(\sX_{t_\IC}, \IZ(1))$, the element $f_{F} \tensor \IC \in L(\sA_{s_\IC})$ produces a Hodge class in $P^2(\sX_{t_\IC}, \IZ(1))$, which by Lefschetz $(1, 1)$-theorem comes from a line bundle $\zeta_{t_\IC} \in \Pic(\sX_{t_\IC})$. Specialize $\zeta_{t_\IC}$ to an element $\zeta \in \Pic(\sX_t)$. Note that $\zeta$ does not depend on the choice of $s_F$ because $c_1(\zeta) \in P^2_\cris(\sX_t/W)$ depends only on $f$. The desired map $L(\sA_s) \to \< \bxi_t \>^\perp \subset \Pic(\sX_t)$ is given by sending $f$ to $\zeta$. 
\end{proof}

Let $V$ be a finite flat extension of $W(\bar{\IF}_p)$. Denote the fraction field of $V$ by $K$ and fix an isomorphism $\bar{K} \iso \IC$. Let $X_V$ over $V$ be a lift of a K3 surface $X$ over $\bar{\IF}_p$ such that $X_V \tensor \IC$ has CM. One can infer from \cite[Theorem 1.1]{Ito} that if $X$ has finite height, then the specialization map $\Pic(X_V \tensor \IC)_\IQ \to \Pic(X)_\IQ$ is an isomorphism. In view of Proposition~\ref{submotive}, there should be a corresponding statement for special endomorphisms:
\begin{proposition}
\label{CMsp}
Let $s \in \shS_{\fK}(G, \Ohm)(\bar{\IF}_p)$ be a point such that $\sA_s$ is non-supersingular. Suppose there exists a lift $s_V$ of $s$ over $V$ such that $s_\IC := s_V \tensor \IC$ is a special point on $\Sh_\fK(G, \Ohm)$. Then the specialization map 
$$ L(\sA_{s_\IC})_\IQ \to L(\sA_{s})_\IQ $$
is an isomorphism. 
\end{proposition}

Here we are implicitly using \cite[Lem.~5.13]{CSpin}, which implies that the specalization map of endomorphisms $\End(\sA_{s_\IC}) \to \End(\sA_s)$ restricts to a map of special endomorphisms $L(\sA_{s_\IC}) \to L(\sA_s)$. 

\begin{proof}
This can be seen as a reinterpretation of \cite[Thm~1.1]{Ito}. To explain the idea, we first assume that there exists $t_V \in \wt{\Mod}_{2d, \fK^\ad, \IZ_{(p)}}(V)$ such that $\rho_{\fK^\ad, \IZ_{(p)}}(t_V)$ is the image of $s_V$.  Set $t = t_V \tensor \bar{\IF}_p$. Since $\sA_{s_V \tensor \IC}$ has CM, so does $X_{t_V \tensor \IC}$ (cf. Lem.~\ref{CMcors}). Then we can conclude by Prop.~\ref{submotive}. Now we sketch how to interpret the computation in \cite[Section 4]{Ito} without appealing to K3 surfaces.\\

\noindent \textit{Step 1: }Recall that $\bL_{B, s_\IC}$ carries a Hodge structure of K3 type. Set $T(\bL_{B, s_\IC}) := (\bL_{B, s_\IC}^{(0 ,0)})^\perp \subset \bL_{B, s_\IC} \tensor \IQ$. By Thm~\ref{Zarhin} and Lem.~\ref{CMcors}, $\End_\Hdg T(\bL_{B, s_\IC}) = E$ for some CM field $E$ such that $T(\bL_{B, s_\IC})$ is one-dimensional over $E$. Let $E_0$ be the maximal totally real subfield of $E$. Thm~\ref{Zarhin}(b) tells us that
$$\MT(\bL_{B, s_\IC}) = \MT(T(\bL_{B, s_\IC})) = \ker (\Nm : \mathrm{Res}_{E/\IQ} \IG_m \to \mathrm{Res}_{E_0/\IQ} \IG_m)$$
Denote this group by $G_0$. \\\\
\textit{Step 2: }By extending $V$ if necessary, we can find a number field $F$ with $E \subset F \subset K = V[1/p]$ such that $s_{K}$ arises from a $F$-valued point $s_F$ on $\Sh_K(G, \Ohm)$. Let $v$ be the finite place above $p$ given by the inclusion $F \subset K$ and let $\fq$ be the place of $E$ below $v$.  Let $F_v$ be the completion of $F$ at $v$ and $\sO_{F_v}$ be the ring of integers of $F_v$ with residue field $k(v)$. Then $s_{F_v} := s_F \tensor F_v$ extends to a $\sO_{F_v}$-valued point $s_{\sO_{F_v}}$ on $\shS_{\fK}(G, \Ohm)$ such that $s_{k(v)} := s_{\sO_{F_v}} \tensor k(v)$ gives the point $s$ when base changed to $\bar{\IF}_p$. To sum up, we have a commutative diagram 
\begin{center}
\begin{tikzcd}
\Spec \bar{\IF}_p \arrow[bend left]{rr}{s}\ \arrow{r}{} \arrow{d}{} & \Spec V \arrow{r}{s_V} \arrow{d}{} & \shS_{\fK}(G, \Ohm) \\
\Spec k(v)\arrow{r}{} & \Spec \sO_{F_v} \arrow{ur}{s_{\sO_{F_v}}}  & \Spec F_v \arrow{l}{} \arrow{u}{}
\end{tikzcd}
\end{center}
The descent of $s$ to $s_{k(v)}$ clearly endows $\bL_{\ell, s}$ with an action by the geometric Frobenius $\Fr \in \Gal(\bar{\IF}_p / k(v))$.  \\\\
\textit{Step 3: }Choose algebraic closures $\bar{F}$ and $\bar{F}_v$ for $F$ and $F_v$ and embeddings $\bar{F} \subset \bar{F}_v \subset \bar{K}$ which are compatible with $F \subset F_v \subset K$. Let $\overline{s_F}$ be the geometric point over $s$ given by $\bar{F}$. Then we have natural isomorphisms $L(\sA_{\overline{s_F}}) \iso L(\sA_{s_\IC}) $ and $\bL_{\IA_f^p, s_\IC} \iso \bL_{\IA^p_f,\overline{s_F}}$. Therefore, if we define $T(\bL_{\IA_f, \overline{s_F}})$ to be the orthogonal complement of the image of $L(\sA_{\overline{s_F}})$ in $\bL_{\IA_f, \overline{s_F}}$, then $\bL_{B, s_\IC} \tensor \IA_f \iso \bL_{\IA_f, \overline{s_F}}$ restricts to an isomorphism $$T(\bL_{B, s_\IC}) \tensor \IA_f \iso T(\bL_{\IA_f, \overline{s_F}}).$$
The $\Gal(\bar{F}/F)$-action on $\bL_{\IA_f,\overline{s_F}}$ restricts to one on $T(\bL_{\IA_f, \overline{s_F}})$. The theory of canonical models tells us that the induced action of $\Gal(\bar{F}/F)$ on $T(\bL_{B, s_\IC}) \tensor \IA_f$ is given by a homorphism $\gamma : \Gal(\bar{F}/F) \to G_0(\IA_f)$. Moreover, we obtain the following commutative diagram 
\begin{center}
\begin{tikzcd}
\IA^\times_F \arrow{r}{\Nm_{\IA_F/\IA_E}} \arrow{d}{\mathrm{art}_F} & \IA_E^\times \arrow{r}{\mathrm{proj}} & \IA^\times_{E, f} \arrow{d}{y \mapsto c(y)y^{-1}} \\
\Gal(F^\ab/F) \arrow{r}{\gamma} & G_0(\IA_f) \arrow{r}{} & G_0(\IA_f)/ G_0(\IQ)
\end{tikzcd}
\end{center}
where $F^\ab$ denotes the maximal abelian extension of $F$, $c$ the complex conjugation on $E \subset \IC$, and $\mathrm{art}_F$ the Artin reciprocity map (cf. \cite[Thm~2.1]{Ito}, \cite[Cor.~3.9.2]{Rizov}, \cite[Thm~12]{Taelman}).\footnote{In these references, the above diagram is stated in terms of K3 surfaces. However, Rizov's theorem \cite[Cor.~3.9.2]{Rizov} is a formal consequence of the rationality of the period map $\rho_{K, \IC}$. The diagram itself comes from the canonical model $\Sh_{\fK^\ad}(G^\ad, \Ohm)$ of $\Sh_{\fK^\ad}(G^\ad, \Ohm)_\IC$.} Here $\mathrm{art}_F$ is normalized so that the images of the uniformizers under the local Artin reciprocity maps act as lifts of geometric Frobenii on unramified extensions. \\\\
\textit{Step 4: }
Let $\pi_v$ be a uniformizer of $F_v$ and $x \in \IA^\times_F$ be the element with $\pi_v$ at $v$ and $1$ at other places. Let $\wt{\sigma} := \mathrm{art}_F(x)$. We use the diagram in step 3 to analyze the characteristic polynomial $f_{\wt{\sigma}}$ of $\wt{\sigma}$ acting on $T(\bL_{B, s_\IC}) \tensor \IQ_\ell$ for $\ell \neq p$. Let $z$ be the image of $x$ under the composition 
$$ \IA^\times_F \stackrel{\Nm}{\to} \IA^\times_{E} \stackrel{\mathrm{pr}}{\to} \IA^\times_{E, f} \stackrel{y \mapsto c(y)y^{-1}}{\to} G_0(\IA_f)  \subset (E \tensor \IA_f)^\times. $$
The commutativity of the diagram in step 3 tells us that there exists $a \in G_0(\IQ)$ such that $\gamma(\wt{\sigma}) = z a$. Then we have that
$$ f_{\wt{\sigma}} = f_{a}^{[E : \IQ(q)]}$$
where $f_a$ is the minimal polynomial of $a \in G_0(\IQ) \subset E^\times$. \\\\ 
\textit{Step 5: }We claim that $\fq$ splits in $E$. By way of contradiction, suppose that $\fq$ does not split in $E$. Then by \cite[Lemma 4.1]{Ito} $a \in E^\times$ is a root of unity. This implies that a power of $\Fr$ acts trivially on $\bL_{\ell, s} \subset \End (H^1(\sA_s, \IQ_\ell))$. By the proof of \cite[Proposition 21]{Charles}, $\sA_s$ must be supersingular. We repeat the argument in \textit{loc. cit.} using our notation for convenience of the reader: By construction of the \'etale local system $\bL_\ell$, there is a $\Fr$-equivariant isomorphism
\begin{align}
\label{equivar}
    \Cl(\bL_{\ell, s}) \iso \End_{\Cl(L_d)} (H^1_\et(\sA_s, \IZ_\ell)).
\end{align}
Here $\Cl(L)$ acts on $H^1_\et(\sA_s, \IZ_\ell)$ from the right. Recall that $\sA_s$ carries a $\IZ/2 \IZ$-grading and a (left) $\Cl(L)$-action. Let $\Cl^+(L)$ (resp. $\Cl^-(L)$) be the even (resp. odd) part of $\Cl(L)$. Similarly, let $\sA_s^+$ (resp. $\sA_s^-$) be the even (resp. odd) part of $\sA_s$. Any invertible element of odd degree in the algebra $\Cl(L)$ induces a quasi-isogeny $\sA^+_s \to \sA_s^-$, so $\sA_s^+$ and $\sA_s^-$ are isogenous. Therefore, it suffices to show that $\sA_s^+$ is supersingular. For $\sA_s^+$, (\ref{equivar}) restricts to an isomorphism 
\begin{align}
    \label{equivar+}
    \Cl^+(\bL_{\ell, s}) \iso \End_{\Cl^+(L_d)} (H^1_\et(\sA_s^+, \IZ_\ell)).
\end{align} If for some $m$, $\Fr^m$ acts trivially on $\End_{\Cl^+(L_d)} (H^1_\et(\sA_s^+, \IZ_\ell))$, then $\Fr^m$ acts through the center of $\Cl^+(\bL_{\ell, s})$. However, $\Cl^+(\bL_{\ell,s})$ is a central simple algebra, so $\Fr^m$ must be a homothety. This forces $\sA_s$ to be supersingular, which contradicts our assumption. Therefore, $\fq$ splits in $E$. By \cite[Lemma 4.1]{Ito} again, none of the roots of $f_{\wt{\sigma}}$ can be a root of unity. \\\\
\textit{Step 6: }Now embed $L(\sA_{s_\IC})$ into $\bL_{\ell, s} \tensor \IQ_\ell$ via the composition of maps
$$ L(\sA_{s_\IC}) \to L(\sA_s) \into \bL_{\ell, s} \tensor \IQ_\ell $$
where the first map is given by specialization. By the smooth and proper base change theorem, $f_{\wt{\sigma}}$ is equal to the characteristic polynomial of $\Fr$ acting on the orthogonal complement of $L(\sA_{s_\IC})$ in $\bL_{\ell, s} \tensor \IQ_\ell$. Since none of the roots of $f_{\wt{\sigma}}$ is a root of unity, the composition
$$ L(\sA_{s_{\IC}}) \tensor \IQ_\ell  \to L(\sA_{s}) \tensor \IQ_\ell \to \varinjlim_{m} (\bL_{\ell, s} \tensor \IQ_\ell)^{\Fr^m = 1}$$
is an isomorphism.  Therefore, $L(\sA_{s_{\IC}})_\IQ \to L(\sA_s)_\IQ$ is an isomorphism. 
\end{proof}

\section{Formal Groups and Cohomology} 
In this section, we review Nygaard-Ogus theory and prove an important lifting lemma which we shall use in Section~\ref{Proofs}. 

\subsection{Formal Brauer Groups}
An abelian variety's crystalline cohomology can be recovered from its $p$-divisible group via Dieudonn\'e modules. There is a similar story for K3 surfaces of finite height, for which the role of $p$-torsion is taken by the formal Brauer group. Let $k$ be an algebraically closed field of characteristic $p > 2$.  

\subsubsection{} Let $X$ be a K3 surface over $R$, where $R$ is an Artinian local ring over $W_n(k) := W(k)/(p^n)$ for some $n$. We denote the formal Brauer group of $X$ by $\BR_X$. Assume that $\BR_{X \tensor_R k}$ has finite height. 

Define a functor $\Psi_{X}$ from the category $\mathsf{Art}_R$ of Artinian local $R$-algebras with residue field $k$ to the category $\mathsf{Ab}$ of abelian groups by setting 
$$ \Psi_{X}(A) = H^2_\fl(X \tensor_R A, \mu_{p^\infty}) $$
where $H^2_\fl$ denotes the second flat cohomology. Set $\sD = H^2_\fl(X \tensor_R k, \mu_{p^\infty})$. By \cite[Lem.~3.1]{NO}, $\sD$ is an abstract $p$-divisible group. We will view it as an \'etale $p$-divisible group over $k$. 

\begin{proposition}
$\Psi_{X}$ defines a $p$-divisible group over $\Spec R$ with connected component $\Psi_{X}^0 = \BR_{X}$ and \'etale part $\Psi_{X}^\et = \sD_R$, where $\sD_R$ denotes the unique $p$-divisible group over $\Spec R$ lifting $\sD/ \Spec k$. 
\end{proposition}
\begin{proof}
This is a corollary to \cite[IV Prop.~1.8]{AM}. See also \cite[Prop.~3.2]{NO}. 
\end{proof}

Recall that Tate established an equivalence between the category of divisible connected formal Lie groups and that of connected $p$-divisible groups over a complete Noetherian local ring with residue field of characteristic $p$ (cf. \cite[Prop.~1]{Tate}). We will freely make use of this equivalence of categories. 

There is a canonical sequence of F-crystals on $\Cris(R/W(k))$ (cf. \cite[Thm~3.20]{NO}): 
 \begin{align}
 \label{rhoeta}
      0 \to \ID(\Psi_X^*) \stackrel{\varrho}{\to} H^2_\cris(X) \stackrel{\theta}{\to} \ID(\BR_X)(-1) \to 0
 \end{align}
For any object $S$ in the nilpotent crystalline site of $R/W(k)$, the above sequence induces a short exact sequence of (Zarsiki) sheaves over $S$. Here $H^2_\cris(X)$ denotes the crystal $R^2 f_{\cris *} \sO_{X/W(k)}$, where $f : X \to \mathrm{Spec\,} R$ is the structure morphism. 

The readers are refered to \cite[Sect.~3]{NO} for the construction of $\varrho$. Here we explain that $\theta$ is constructed out of $\varrho$ by duality: Let $\iota : \ID(\BR_X^*) \to \ID(\Psi_X^*)$ be the map induced by the inclusion $\BR_X \to \Psi_X$ and let $\kappa : \ID(\BR_X^*)^\vee \stackrel{\sim}{\to} \ID(\BR_X)(1)$ be the isomorphism induced by the canonical pairing $\ID(\BR_X^*) \times \ID(\BR_X) \to \sO_{R/W(k)}(-1)$. Poincar\'e duality on $H^2_\cris(X)$ induces a map $\varrho^\vee : H^2_\cris(X) \to \ID(\Psi_X^*)^\vee(-2)$. The map $\theta$ is defined as the composition 
\begin{align}
\label{duality}
    H^2_\cris(X) \stackrel{\varrho^\vee}{\to} \ID(\Psi_X^*)^\vee(-2) \stackrel{\iota^\vee(-2)}{\to} \ID(\BR_X^*)^\vee(-2) \stackrel{\kappa(-2)}{\to} \ID(\BR_X)(-1).  
\end{align}

\subsection{Constructing Liftings} 
For the rest of section $4$, assume that $p \ge 5$. Let $X$ be a K3 surface over $k$ with $\BR_X$ of finite height. There is a unique splitting of (\ref{rhoeta}) for $X$, so that there is a canonical identification (cf. the proof of \cite[Prop.~5.4]{NO})
\begin{equation}
    \label{decomp}
    H^2_\cris(X/W(k)) = \IK(\BR_X^*) := \ID(\BR_{X}^*) \oplus \ID(\sD^*) \oplus \ID(\BR_{X})(-1).
\end{equation}
The construction of $\IK(\BR_X)$ explains the slope filtration on $H^2_\cris(X/W(k)) $:
$ H^2_\cris(X/W(k)) _{< 1} = \ID(\BR_{X}^*) , H^2_\cris(X/W(k)) _{= 1} =  \ID(\sD^*)$ and $H^2_\cris(X/W(k)) _{> 1} = \ID(\BR_{X})(-1)$. 

Let $K$ be a finite extension of $K_0 := W(k)[1/p]$ and let $V = \sO_K$ be the ring of integers of $K$. Suppose $\shG_V$ is a lifting of $\BR_{X}$ over $V$. Then $\shG_V$ determines a filtration $\Fil_{\shG_V} \subset \ID(\BR_X) \tensor_{W(k)} K$. We can use $\Fil_{\shG_V}$ to define on $\IK(\BR_X) \tensor_{W(k)} K$ a two-step filtration $\Fil_{\shG_V}^\bullet$:
$$ \Fil^2_{\shG_V} \subset \Fil^1_{\shG_V} \subset \Fil^0_{\shG_V} = \IK(\BR_X) \tensor_{W(k)} K $$
by setting $\Fil^2_{\shG_V} := \Fil_{\shG_V}(-1) \subset \ID(\BR_X)(-1) \tensor_{W(k)} K$ and $\Fil^1_{\shG_V} := (\Fil^2_{\shG_V})^\perp$ with respect to the bilinear pairing on $\IK(\BR_X) \tensor_W K = H^2_\cris(X/W(k)) \tensor_{W(k)} K$.  

The following result of Nygaard-Ogus allows us to choose a lift a K3 surface by choosing one for its formal Brauer group.
\begin{proposition}
\label{NyOMain}
For each $\shG_V$ lifting $\what{\Br}_{X}$ as above, there exists a K3 surface $X_V$ over $V$ lifting $X$ with the following properties: (a) $\Psi_{X_V} \iso \shG_V \oplus \shD_V$, where $\shD_V$ is the unique $p$-divisible group over $\mathrm{Spec\,} V$ lifting $\sD$. (b) The natural map $\Pic(X_{V}) \to \Pic(X)$ is an isomorphism. (c) The Berthelot-Ogus isomorphism $\sigma_\cris : H^2_\dR(X_V/V) \tensor K \stackrel{\sim}{\to} H^2_\cris(X/W(k)) \tensor K = \IK(\BR_X) \tensor K$ respects the pairings and takes the Hodge filtration $\Fil^\bullet_\Hdg$ to the filtration $\Fil^\bullet_{\shG_V}$. 
\end{proposition}
\begin{proof}
This is \cite[Prop.~5.5]{NO}, except that we do not tensor $\sigma_\cris$ with $\bar{K}_0$ and let the crystalline-Weil group $W_\cris(\bar{K}_0)$ (cf. \cite[Def.~4.1]{BO}) act on both sides. Here $\bar{K}_0$ is an algebraic closure of $K_0$ containing $K$ as a subfield. Indeed, \cite[Prop.~5.5]{NO} tells us that there is lifting $X_V$ such that the isomorphism
\begin{align}
\label{crysWeil}
    H^2_\cris(X_V/V) \tensor_V \bar{K}_0 \stackrel{\sim}{\to} H^2_\cris(X/W(k)) \tensor_{W(k)} \bar{K}_0 = \IK(\BR_X) \tensor_{W(k)} \bar{K}_0
\end{align}
takes $\Fil^\bullet_\Hdg$ to $\Fil_{\shG_V}^\bullet$ and respects the pairings and the action of the crystalline-Weil group $W_\cris(\bar{K}_0)$. The above isomorphism is constructed as $\sigma_\cris \tensor_V \mathrm{id}_{\bar{K}_0}$ (cf. \cite[Thm~4.2]{BO} and its proof). As explained in \textit{loc. cit.}, this isomorphism is independent of the choice of $K$ in the following sense: Suppose that $K' \subset \bar{K}_0$ is a finite extension of $K$ and $V'$ is the ring of integers of $K'$. By \cite[Prop.~2.7]{BO}, if we form the Berthelot-Ogus isomorphism $\sigma_\cris' : H^2_\dR(X_V \tensor_V V' / V') \tensor_{V'} K' \stackrel{\sim}{\to} H^2_\cris(X/W(k)) \tensor_{W(k)} K'$ for $X_V \tensor_V V'$, then $\sigma'_\cris \tensor_{V'} \bar{K}_0 = \sigma_\cris \tensor_V \bar{K}_0$ when we identify $H^2_\dR(X_V/V) \tensor_V K'$ with $H^2_\dR(X_V \tensor_V V'/V') \tensor_{V'} K'$. 

The proof of \cite[Prop.~5.5]{NO} mentioned that the arguments for \cite[Prop.~1.8]{Nygaard} show that the natural map $\Pic(X_V) \to \Pic(X_V \tensor V/(p))$ is an isomorphism. In fact, the arguments for \cite[Prop.~1.8]{Nygaard} imply the stronger statement (b).\footnote{We thank an anonymous referee for pointing this out.} Here we present an alternative proof in order to be a little more self-contained. To begin, we give some detail on the construction of $X_V$. Let $t$ be a uniformizer of $V$ and let $e$ be the ramification degree of $V$ over $W(k)$. For each number $n = 0, 1, \cdots, e - 1$, set $A_n := V/ (t^{n + 1})$, $\shG_n := \shG_V \tensor A_n$ and $\shD_n := \shD_V \tensor A_n$. Set $\IK(\shG_n) := \ID(\shG_n^*) \oplus \ID(\shD_n^*) \oplus \ID(\shG_n)(-1)$. The paragraph above \cite[Prop.~5.4]{NO} has explained how to put a K3 crystal structure (cf. \cite[Def.~5.1]{NO}) on $\IK(\shG_n)$. Since $\shG_n$ is a deformation of $\shG_0 = \Br_{X}$, \cite[Prop.~5.4]{NO} gives us a deformation $X_n$ of $X_0 = X$ to $A_n$, together with isomorphisms $h_n : \Br_{X_n} \iso \shG_n$, $j_n : \Psi_{X_{n}} \iso \shG_{n} \oplus \shD_n$ and an isomorphism of K3 crystals $i_n : H^2_\cris(X_n) \iso \IK(\shG_n)$. One can check from the proofs of \cite[Thm.~5.3, Prop.~5.4]{NO} that each $X_{n + 1}$ is constructed as a deformation of $X_n$, and $(h_{n + 1}, j_{n + 1}, i_{n + 1})$ restricts to $(h_n, j_n, i_n)$. Finally, note that $(V, (p))$ has a natural PD structure. $X_V$ is constructed as the deformation of $X_{e - 1}$ with the desired Hodge filtration on $H^2_\dR(X_V/V) \iso H^2_\cris(X_{e - 1})_V$. In particular, $X_V \tensor V/(p) = X_{e - 1}$.

We first show that for $n \le e - 2$, every line bundle $\zeta \in \Pic(X_n)$ lifts to $X_{n + 1}$. This implies that $\Pic(X_{e - 1}) \to \Pic(X)$ is an isomorphism. Note that there is a canonical isomorphism $H^2_\cris(X_n)_{A_{n + 1}} \iso H^2_\dR(X_{n + 1}/A_{n + 1})$. By \cite[Prop.~1.12]{Ogus}, it suffices to check that $c_{1, \cris}(\zeta)_{A_{n + 1}}$ lies in $\Fil^1 H^2_\dR(X_{n + 1}/A_{n + 1})$, which is equal to $[\Fil^2 H^2_\dR(X_{n + 1}/A_{n + 1})]^\perp$ in our case. It follows from the K3 crystal structure on $\IK(\shG_{n + 1})$ that $\ID(\shD^*_{n + 1})$ is orthogonal to $\ID(\shG_{n + 1})(-1)$ and $\Fil^2 H^2_\dR(X_{n + 1}/A_{n + 1})$ lies (via $i_{n + 1}$) in the $\ID(\shG_{n + 1})(-1)_{A_{n + 1}}$ component. Therefore, it suffices to show that $c_{1, \cris}(\zeta)_{A_{n + 1}} \in \ID(\shD^*_{n + 1})_{A_{n + 1}}$. 

Let $R$ a finite flat extension of $W(k)$ with ramification degree $n + 2$ and fix an isomorphism $R/ (p) \iso A_{n + 1}$. Let $L$ be the fraction field of $R$. Let $\pi$ be a uniformizer of $R$ which is sent to the image of $t \in V$ in $A_{n + 1}$. Then $R/ (p) \iso A_{n + 1}$ descends to an isomorphism $R/(\pi^{n + 1}) \iso A_{n}$. By \cite[Lem.~3.9]{BO}, the ideal $(\pi^{n + 1}) \subset R$ has a canonical PD structure. By applying \cite[Cor.~2.2]{BO} to $X_n$ and $X \tensor A_n$, we obtain a canonical isomorphism $H^2_\cris(X_n)_R \tensor_R L \iso H^2_\cris(X/W) \tensor_W L$. By \cite[Cor.~3.6]{BO}, this isomorphism is compatible with the Chern class maps, so that $c_{1, \cris}(\zeta)_R \tensor 1$ is sent to $c_{1, \cris}(\zeta_0) \tensor 1$, where $\zeta_0$ is the restriction of $\zeta$ to $X_0 = X$. Similarly, one slightly adapts the proof of \cite[Prop.~3.14]{BO} to obtain a natural isomorphism $\IK(\shG_n)_R \tensor_R L \iso \IK(\BR_X)_W \tensor_W L$ which respects their three components. This isomorphism is compatible with $H^2_\cris(X_n)_R \tensor_R L \iso H^2_\cris(X/W) \tensor_W L$. Finally, since $c_{1, \cris}(\zeta_0) \in H^2_\cris(X/W)$ lies in the $\ID(\sD^*)_W$ component, $c_{1, \cris}(\zeta)_R$ lies in the $\ID(\shD_n^*)_R$ component. Therefore, $c_{1, \cris}(\zeta)_{A_{n + 1}}$, which is the mod $p$ reduction of $c_{1, \cris}(\zeta)_R$, lies in $\ID(\shD^*_n)_{A_{n + 1}} = \ID(\shD^*_{n + 1})_{A_{n + 1}}$. 

A little extension of the above argument shows that $\Pic(X_V) \to \Pic(X_{e - 1})$ is an isomorphism. Let $\zeta_{e - 1}$ be an element in $\Pic(X_{e - 1})$. Recall the natural isomorphism $H^2_\cris(X_{e - 1})_V \iso H^2_\dR(X_V/V)$. By the above argument, we know that $c_{1, \cris}(\zeta_{e - 1})_V$ lies in $\ID(\shD_{e - 1}^*)_V$. By the construction of $X_V$, $\Fil^2 H^2_\dR(X_V)$ lies in the $\ID(\shG_{e - 1})(-1)_V$ component of $H^2_\cris(X_{e - 1})_V$. Hence $c_{1, \cris}(\zeta_{e - 1})_V$ is orthogonal to $\Fil^2 H^2_\dR(X_V/V)$ and lies in $\Fil^1 H^2_\dR(X_V/V)$. By \cite[Prop.~1.12]{Ogus} again, $\zeta_{e - 1}$ lifts to $X_V$. 
\end{proof}

\begin{remark}
\label{ampleliftstoample}
Let $X$ and $X_V$ be as above. It follows from \cite[Thm~1.2.17]{Laz} that the ample cone of $\Pic(X)$ lies inside the ample cone of $\Pic(X_K)$. By a theorem of Kleiman (cf. \cite[Thm~1.4.23]{Laz}), the nef cone of $X$ or $X_K$ is the closure of the ample cone, so the nef cone of $\Pic(X)$ also lies inside that of $\Pic(X_K)$. Finally, a nef line bundle is big if and only if its top self-intersection is strictly positive (cf. \cite[Thm~2.2.16]{Laz}), so the lift of a big and nef divisor on $X$ to $X_{K}$ is also big and nef. 
\end{remark}

\begin{lemma}
\label{LiftFG}
Let $\shG$ be a one dimensional formal group of height $h$ over $k = \bar{\IF}_p$ and $\alpha : \shG \to \shG$ be an isogeny. There exists a finite flat extension $V$ over $W$ of ramification index $\le h$ and a lift $\shG_V$ of $\shG$ to $V$ such that $\alpha$ lifts to an isogeny $\alpha_V : \shG_V \to \shG_V$. 
\end{lemma}
\begin{proof}
View $\shG$ as a $\IZ_p$-module (for this terminology, see \cite[1.1]{Wewers}). Set $M = \IQ_p(\alpha)$ and let $\sO_M$ be its ring of integers. Let $V:= \what{\sO}_M^\ur$ be the completion of the maximal unramified extension of $\sO_M$. Let $e := [M : \IQ_p]$. Then as an $\sO_M$-module $\shG$ has height $h/e$. The universal deformation space of an $\sO_M$-module of height $h/e$ is isomorphic to $\what{\sO}_M^\ur[[t_1, \cdots, t_{h/e - 1}]]$. Hence we may obtain a lift $\shG_V$ of $\sO_M$-module $\shG$ to $\what{\sO}_M^\ur$ simply by setting $t_1 = \cdots = t_{h/e - 1} = 0$. Clearly $\shG_V$ carries an action of $\alpha$ and its special fiber is isomorphic to $\shG$ by construction. 
\end{proof}

Now we can prove a lifting lemma which we shall use in Section 5.
\begin{lemma}
\label{workhorce}
Let $X$ and $X'$ be two K3 surfaces of finite height over $k = \bar{\IF}_p$. For any isometry of F-isocrystals
$$ \phi : H^2_\cris(X/W) \tensor K_0 \stackrel{\sim}{\to} H^2_\cris(X'/W) \tensor K_0$$
we can find a finite flat extension $V$ of $W$ and lifts $X_V, X_V'$ of $X, X$ over $V$ with the following properties:
\begin{enumerate}[label=(\alph*)]
    \item The natural maps $\Pic(X_V) \to \Pic(X)$ and $\Pic(X'_V) \to \Pic(X')$ are isomorphisms.
    \item The map $H^2_\dR(X_V/V) \tensor K \stackrel{\sim}{\to} H^2_\dR(X'_V/W) \tensor K$ induced by $\phi$ via the Berthelot-Ogus isomorphism preserves the Hodge filtrations, where $K = V[1/p]$. 
\end{enumerate}
When $X = X'$, $X_V$ and $X_V'$ can be taken to be the same lifting. 
\end{lemma}

\begin{proof}
Since $\phi$ respects the Frobenius action, under the identification (\ref{decomp}) $\phi$ restricts to an isomorphism of isocrystals $\ID(\BR_X)(-1) \tensor K_0 \stackrel{\sim}{\to} \ID(\BR_{X'})(-1) \tensor K_0$. Therefore, $\what{\Br}_{X'}$ and $\what{\Br}_{X}$ must have the same height, and hence they are abstractly isomorphic. Let $\phi_{\mathrm{Br}}$ be the quasi-isogeny $\Br_{X'} \to \Br_X$ induced by $\phi$. 

By Lem.~\ref{LiftFG}, we can find a finite flat extension $V$ of $W$ and lifts $\shG_V$ and $\shG_V'$ of $\what{\Br}_{X}$ and $\what{\Br}_{X'}$ such that $\phi_{\mathrm{Br}} : \what{\Br}_{X'} \to \what{\Br}_{X}$ lifts to an quasi-isogeny $\phi_{\mathrm{Br}, V} : \shG'_V \to \shG_V$. The desired $X_V$ and $X_V'$ can be taken to be lifts given by Prop.~\ref{NyOMain} such that the Berthelot-Ogus isomorphisms $H^2_\dR(X_V/V) \tensor_V K \stackrel{\sim}{\to} \IK(\BR_X) \tensor_{W} K$ and $H^2_\dR(X'_V/V) \tensor_V K \stackrel{\sim}{\to} \IK(\BR_{X'}) \tensor_{W} K$ send the Hodge filtrations on $H^2_\dR(X_V/V) \tensor K$ and $H^2_\dR(X'_V/V) \tensor K$ to $\Fil^\bullet_{\shG_V}$ and $\Fil^\bullet_{\shG_V'}$ respectively. 

When $X = X'$, we identify $\BR_X$ and $\BR_{X'}$ and lift $\phi$ to a quasi-isogeny $\shG_V \to \shG_V$. Then of course $X_V$ and $X_V'$ can be taken to be the same lifting. 
\end{proof}

\begin{remark}
Not surprisingly, if $X$ and $X'$ are ordinary, we can always take $V$ to be $W$ and $X_V, X_V'$ to be the canonical liftings. The reason is that the canonical lifting of an ordinary K3 surface induces a filtration which coincides with the slope filtration (\cite[Lem.~1.9]{Yu}). 
\end{remark}

\section{Proofs of Theorems}
\paragraph{Convention} The universal family over $\wt{\Mod}_{2d, \fK^\ad_p, \IZ_{(p)}}$ is denoted by $(\sX, \bxi)$. For any scheme $T$, we denote the image of a point $t \in \wt{\Mod}_{2d, \fK^\ad_p, \IZ_{(p)}}(T)$ in $\shS_{\fK^\ad_p}(G^\ad, \Ohm)(T)$ still by $t$. Similarly, we will also make use of the isomorphisms in Prop.~\ref{compareCoh} with the symbol for period morphism suppressed. Denote the fiber of $(\sX, \xi)$ over $t$ by $(\sX_t, \bxi_t)$. 

We first give a simple lemma on correspondences:

\begin{lemma}
\label{corr}
Let $Y, Y'$ be two algebraic surfaces over a field. Every morphism $\psi: \NS(Y)_\IQ \to \NS(Y')_\IQ$ is induced by a correspondence on $Y \times Y'$.
\end{lemma}
\begin{proof}
Recall that by the Hodge index theorem, the intersection pairing on $\NS(Y)$ is non-degenerate. Let $e_1, \cdots, e_n$ be a basis for $\NS(Y)_\IQ$ and let $e_1^*, \cdots, e_{n}^* \in \NS(Y)_\IQ$ be a dual basis, i.e., under the intersection pairing $e_i^* \cdot e_i = 1$ and $e_i^* \cdot e_j = 0$ for $i \neq j$. Let $f_1, \cdots, f_{n'}$ be a basis for $\NS(Y')_\IQ$. For each $i$, let $E_i, E_i^*, F_i$ be formal $\IQ$-linear combinations of curves on $Y$ or $Y'$ representing the classes $e_i, e_i^*, f_i$. If $\psi : \NS(Y)_\IQ \to \NS(Y')_\IQ$ sends $e_i$ to $\sum_{j = 1}^{n'} a_{ij} f_j$ for $a_{ij} \in \IQ$, then the desired correspondence is given by $\sum_{i = 1}^n \sum_{j = 1}^{n'} a_{ij} E_i^* \times F_j$ on $Y \times Y'$. 
\end{proof}

\label{Proofs}
\subsection{Proof of Proposition~\ref{mainTh}} 
Proposition~\ref{mainTh} is a direct consequence of the following more precise statement:
\begin{proposition}
\label{main}
Assume $p \nmid d$ and $p \ge 13$. Let $t, t' \in \wt{\Mod}_{2d, \fK^\ad_p, \IZ_{(p)}}(\bar{\IF}_p)$ be two points. Let $W := W(\bar{\IF}_p)$ and $K_0 := W[1/p]$. Suppose the images of $t, t'$ in $\shS_{\fK^\ad_p}(G^\ad, \Ohm)(\bar{\IF}_p)$ lift to points $s, s'\in \shS_{\fK_p}(G, \Ohm)(\bar{\IF}_p)$ respectively. For each CSpin-isogeny $\psi : \sA_s \to \sA_{s'}$, there exists an isogeny $\phi : \sX_t \isog \sX_{t'}$ which sends $\mathrm{ch}_*(\bxi_{t'})$ to $\mathrm{ch}_*(\bxi_{t})$ for $* = \cris, \et$ such that the following diagrams commute
\begin{center}
\begin{tikzcd}
\bL_{\cris, t'} \tensor K_0 \arrow{r}{\text{conj. by }\psi^*_\cris} \arrow{d}{} & \bL_{\cris, t} \tensor_W K_0 \arrow{d}{} \\
P^2_\cris(\sX_{t'}/W) \tensor K_0(1) \arrow{r}{\phi^*_\cris} & P^2_\cris(\sX_t/W) \tensor K_0(1).
\end{tikzcd}
\begin{tikzcd}
\bL_{\IA_f^p, t'} \arrow{r}{\text{conj. by }\psi^*_\et} \arrow{d}{} & \bL_{\IA^p_f, t} \arrow{d}{} \\
P^2_\et(\sX_{t'}, \IA_f^p)(1) \arrow{r}{\phi^*_\et} & P^2_\et(\sX_t, \IA_f^p)(1)
\end{tikzcd}
\end{center}
\end{proposition}

\begin{proof}
We treat the supersingular case and the finite height case separately. \\
\textbf{Supersingular case: }The map $\psi$, as a CSpin-isogeny, clearly preserves special endomorphisms, so it induces an isometry $L(\sA_{s'})_\IQ \stackrel{\sim}{\to} L(\sA_{s})_\IQ$. By Thm~\ref{submotive}, we obtain an isometry $i^\perp: \< \bxi_{t'} \>^\perp \stackrel{\sim}{\to} \< \bxi_{t} \>^\perp$. Here the orthogonal complements $\< \bxi_{t} \>^\perp$ and $\< \bxi_{t'} \>^\perp$ are taken inside $\NS(\sX_t)_\IQ$ and $\NS(\sX_{t'})$ respectively. We may extend $i^\perp$ to an isometry $i : \NS(\sX_{t'})_\IQ \stackrel{\sim}{\to} \NS(\sX_{t})_\IQ$ by sending $\bxi_{t'}$ to $\bxi_{t}$. By Lemma~\ref{corr}, $i$ is given by a correspondence $\phi : \sX_t \isog \sX_{t'}$. On the other hand, $\< \bxi_t \>^\perp$ and $\< \bxi_{t'} \>^\perp$ span all of $P^2_\et(\sX_t, \IA^p_f)$ and $P^2_\et(\sX_{t'}, \IA^p_f)$, the induced map $\phi_\et^* : P^2_\et(\sX_{t'}, \IA^p_f) \to P^2_\et(\sX_t, \IA^p_f)$ is completely determined and has to agree with the map induced by $\psi$. The argument for crystalline cohomology is the same.\\\\
\textbf{Finite height case:} By Lem.~\ref{workhorce} and Rmk~\ref{ampleliftstoample}, for some finite flat extension $V$ of $W$ and $K:= V[1/p]$, we may choose quasi-polarized K3 surfaces $(X_V, \xi_V)$ and $(X_V', \xi_V')$ over $V$ which lift $(\sX_t, \bxi_{t})$ and $(\sX_{t'}, \bxi_{t'})$ such that 
\begin{equation}
\label{conjcris}
\text{conj. by } \psi^*_\cris : P^2_\cris(\sX_{t'}/W) \tensor_W K \to P^2_\cris(\sX_t/W) \tensor_W K
\end{equation} 
preserves the Hodge filtrations induced by $(X_V, \xi_V)$ and $(X_V', \xi_V')$. Note that the Berthelot-Ogus isomorphism $H^2_\cris(\sX_t/W) \tensor_W K \iso H^2_\dR(X_V/V) \tensor_V K$ restricts to an isomorphism $$P^2_\cris(\sX_t/W) \tensor_W K \iso P^2_\dR(X_V/V) \tensor_V K$$ and the same holds for $X_V'$ and $\sX_{t'}$. We can then choose lifts $t_V, t'_V : \Spec V \to \wt{\Mod}_{2d, \fK_p^\ad, \IZ_{(p)}}$ of $t, t'$ such that $(\sX_{t_V}, \bxi_{t_V})$ and $(\sX_{t_V'}, \bxi_{t_V'})$ can be respectively identified with $(X_V, \xi_V)$ and $(X_V', \xi_V')$.

Since the map $\shS_{\fK_p}(G, \Ohm) \to \shS_{\fK^\ad_p}(G^\ad, \Ohm)$ is pro-\'etale, we can lift the $V$-valued points $t_V, t'_V$ to $s_V, s'_V$ on $\shS(G, \Ohm)$ such that $s, s'$ are special points of $s_V, s_V'$. Note that under the inclusion $P^2_\dR(\sX_{t_V}(1)) \subset \End( H^1_\dR(\sA_{s_V}))$, the Hodge filtration on $H^1_\dR(\sA_{s_V})$ is given by 
$$ \F^1 H^1_\dR(\sA_{s_V}) = \ker \F^1 P^2_\dR(\sX_{t_V})(1). $$

Therefore, $\psi^*_\cris \tensor K : H^1_\cris(\sA_{s'}/W) \tensor_W K  \stackrel{\sim}{\to} H^1_\cris( \sA_{s}/W) \tensor_W K$ 
preserves the Hodge filtrations induced by $\sA_{s_V}$ and $\sA_{s'_V}$. By \cite[Thm~3.15]{BO}, $\psi$ lifts to a quasi-isogeny $ \psi_V : \sA_{s_V} \to \sA_{s'_V}$. Choose an isomorphism $\bar{K} \iso \IC$, which induces an embedding $V \subset \IC$. We claim that $\psi_V \tensor \IC$ respects the $\IZ/2\IZ$-grading, $\Cl(L)$-action and sends $\bpi_{B, s_V \tensor \IC}$ to $\bpi_{B, s'_V \tensor \IC}$. In particular, $\psi_V \tensor \IC$ is a CSpin-isogeny. Indeed, we can check this by looking at the maps induced by $\psi_V \tensor \IC$ on $\ell$-adic cohomology for any $\ell \neq p$, so that the conclusion follows from the smooth and proper base change theorem and the assumption that $\psi$ is a CSpin-isogeny. Now $\psi_V \tensor \IC$ induces by conjugation a Hodge isometry $$P^2(\sX_{t_V' \tensor \IC}, \IQ) \stackrel{\sim}{\to} P^2(\sX_{t_V \tensor \IC}, \IQ).$$
Extend the above map to full Hodge structures by sending the class of $\bxi_{t'}$ to that of $\bxi_{t}$. Buskin's result tells us that this Hodge isometry is given by an isogeny $\sX_{t_V \tensor \IC} \isog \sX_{t'_V \tensor \IC}$. Now we can complete the proof by specializing this correspondence to $\phi : \sX_t \isog \sX_{t'}$. By the compatibility between specialization and cycle class maps, the diagrams in the statement of the proposition commute. 
\end{proof}

\subsection{Proof of Theorem~\ref{mainCor}}
\begin{lemma}
\label{conjugate}
Let $k$ be a field, $M$ be a vector space over $k$ and $G \subset GL(M)$ be a closed reductive subgroup. Let $\mu, \mu' : \IG_m \to G$ be two cocharacters. If $\mu$ and $\mu'$ induce the same filtration on $V$, then they are conjugate under $G(k)$. 
\end{lemma}
\begin{proof}
This can be extracted from the proof of \cite[Lem.~1.1.9]{Modp}. We present the argument for readers' convenience. Let $\Fil^\bullet M$ be the filtration on $M$ induced by $\mu, \mu'$ and let $P \subset G$ denote the parabolic subgroup which respects $\Fil^\bullet$. Let $U \subset P$ be the subgroup which acts trivially on associated graded vector space $\mathrm{gr}^\bullet M$, so that $U$ is the unipotent radical of $P$. Since $\mu, \mu'$ induce the same filtration, they induce the same grading on $\mathrm{gr}^\bullet M$, which means that their compositions $\IG_m \stackrel{\mu, \mu'}{\to} P \to P/U$ are equal. Let $L, L'$ be the centralizers of $\mu, \mu'$ in $G$. Then $L, L'$ are Levi subgroups of $P$, so $L = u L' u^{-1}$ for some $u \in U(k)$. The cocharacter $\mu'' := u \mu' u^{-1} : \IG_m \to L$ and $\mu$ induce the same cocharacter under the projection $L \stackrel{\sim}{\to} P/U$. Hence $\mu'' = \mu$. 
\end{proof}

Now let $k$ be a perfect field of characteristic $p$ and let $(D, \varphi, \<-, -\>)$ be a K3 crystal over $k$. Recall that the Frobenius action $\varphi$ gives an abstract Hodge filtration $\Fil^\bullet_\varphi$ on $D/pD$: 
$$ \Fil^i_\varphi(D/pD) := \varphi^{-1}(p^i D) \mod p $$

By Mazur-Ogus inequality (cf. \cite[Thm~8.26]{BO}), if $(D, \varphi)$ is given by $H^2_\cris(X/W(k))$ for some K3 surface $X$ over $k$, then $\Fil^\bullet_\varphi$ agrees with the Hodge filtration on $$H^2_\cris(X/W(k))/ p  H^2_\cris(X/W(k)) \iso H^2_\dR(X/k). $$

Let $(D, \varphi)$ be a K3 crystal. Assume that $k$ is algebraically closed. Set $K_0 := W(k)[1/p]$. Suppose that the quadratic lattice $D$ is given by $N \tensor_{\IZ_p} W(k)$ for some self-dual quadratic lattice $N$ over $\IZ_p$ and $\varphi$ is given by $p \bar{b} \sigma$ for some $\bar{b} \in \SO(N \tensor_{\IZ_p} K_0)$. For the following two lemmas, we write $\SO$ for the group scheme $\SO(N)$ over $\IZ_p$. The lemma below is an analogue of \cite[Lem.~1.1.12]{Modp}.
\begin{lemma}
\label{prodcocharacter}
\begin{enumerate}[label=(\alph*)]
    \item There exists a cocharacter $\mu : \IG_{m} \to \SO_{W(k)}$ such that $$\bar{b} \in \SO(W(k)) \mu(p) \SO(W(k))$$ and $\sigma^{-1}(\mu)$ gives the filtration $\Fil^\bullet_\varphi(1)$. 
    \item $\Fil^2_\varphi$ is an isotropic direct summand of $D / p D$ of dimension $1$. 
\end{enumerate}
\end{lemma}

\begin{proof}
(a) By the Cartan decomposition, there exists a $\SO_{W(k)}$-valued cocharacter $\mu'$ and $h_1, h_2 \in \SO_{W(k)}$ such that $\bar{b} = h_1 \mu'(p) h_2 = h_1 h_2 \mu(p)$ for $\mu = h_2^{-1} \mu' h_2$. We check that $\sigma^{-1}(\mu)$ gives the filtration $\Fil^\bullet_\varphi(1)$. The condition that $p^2 D \subset \varphi(D)$ tells us that the decomposition defined by $\sigma^{-1} (\mu)$ takes the form $D = D_1 \oplus D_0 \oplus D_{-1}$ where $D_i = \{ d\in D : \sigma^{-1} (\mu)(z) \cdot d = z^i d \}$. Now we verify that $D_1 / p D_1 = \Fil^2_\varphi(D/pD)$, or equivalently, $D_1$ and $\varphi^{-1}(p^2 D)$ have the same image modulo $p$. One can easily check by definitions $D_1$ is a sub $W(k)$-module of $\varphi^{-1}(p^2 D)$. Conversely, we need to show for every $d \in D$ such that $\varphi(d) \in p^2 D$, $d$ is congruent to an element of $D_1/ p D_2$. Let $d = d_1 + d_0 + d_{-1}$ be the decomposition such that $d_i \in D_i$. Since $\varphi(d) = p \bar{b} \sigma(d) = p h_1 h_2 \mu(p) \sigma(d) \in p^2 D$, $d_0 \equiv d_{-1} \equiv 0 \mod p$. Therefore, $(d\mod p) \in D_1/pD_1$. We can check that $\Fil^1_\varphi(D/pD)$ is equal to $(D_1\oplus D_0)$ modulo $p$ similarly.

(b) The condition that $\mathrm{rank\,} \varphi \tensor_{W(k)} k = 1$ implies that $\mathrm{rank\,} D_{-1} = 1$. Since $D_1 = (D_{-1})^\vee$, $\mathrm{rank\,} D_1 = 1$.
\end{proof} 

\begin{lemma}
\label{crysConj}
Suppose $(D', \varphi')$ is another K3 crystal with $D' \iso N \tensor W(k)$ and there is an embedding $\iota : D' \into D \tensor K_0$ which respects the Frobenius action by $\varphi'$ and $\varphi \tensor 1$. If $g \in \SO(K_0)$ is an element such that $g(D) = \iota(D')$, then $g^{-1} \bar{b} \sigma(g) \in \SO(W(k)) \mu(p) \SO(W(k)) $. 
\end{lemma}
\begin{proof}
We may assume that $D' = N \tensor W(k)$ and $\iota$ is given by restricting the domain of $g : N \tensor K_0 \stackrel{\sim}{\to} N \tensor K_0$ to $N \tensor W(k)$. The action of $\varphi' \tensor 1$ on $D' \tensor K_0$ is then given by $p \bar{b}' \sigma$, where $\bar{b}' = g^{-1} \bar{b} \sigma(g)$. 

Now we use the assumption that $(D', \varphi')$ is a K3 crystal to deduce that $g^{-1} \bar{b} \sigma(g) \in \SO(W(k)) \mu(p) \SO(W(k))$. Let $\mu'$ be a cocharacter of $\SO_{W(k)}$ such that $\bar{b}'$ belongs to $\SO(W(k)) \mu'(p) \SO(W(k))$ and $\sigma^{-1}(\mu')$ gives the filtration $\Fil^\bullet_{\varphi'}(1)$ on $D' \tensor k$. Clearly it suffices to show that $\mu$ and $\mu'$ are conjugate by an element in $\SO(W(k))$. Note that the filtrations on $N \tensor k$ induced by $\sigma^{-1}(\mu)$ and $\sigma^{-1}(\mu')$ are both of the form $0 = \Fil^{-2} \subset \Fil^{1} \subset \Fil^0 \subset \Fil^{-1} = N \tensor k$ with $\Fil^0 = (\Fil^1)^\perp$ and $\dim_k \Fil^1 = 1$. This implies that reductions of $\sigma^{-1}(\mu)$ and $\sigma^{-1}(\mu')$ modulo $p$ are conjugate by an element in $\SO(k)$. By \cite[IX 3.3]{DG}, $\sigma^{-1}(\mu)$ and $\sigma^{-1}(\mu')$, and hence $\mu$ and $\mu'$, are conjugate by element of $\SO(W(k))$.  
\end{proof}

\noindent \textbf{Proof of Theorem~\ref{mainCor}.} Let $(X, \xi)$ be the quasi-polarized K3 surface of Thm~\ref{mainCor}. Clearly it suffices to prove the case when $\xi$ is primitive. Let $t \in \wt{\Mod}_{2d, \fK^\ad_p, \IZ_{(p)}}(\bar{\IF}_p)$ be a point such that $(\sX_t, \bxi_t)$ can be identified with $(X, \xi)$. Let $L = L_d$ and set up period morphisms as in section~\ref{period}. Let $s \in \shS_{\fK_p}(G, \Ohm)(\bar{\IF}_p)$ be a lift of the image of $t$ in $\shS_{\fK^\ad_p}(G^\ad, \Ohm)(\bar{\IF}_p)$.
As in section~\ref{expisog}, fix an isomorphism $H \tensor W \iso H^1_\cris(\sA_s/W)$ of $\IZ/2 \IZ$-graded right $\Cl(L)$-modules which send $\pi$ to $\bpi_{\cris, s}$ and construct maps $\iota_s : X_p \times X^p \to \shS_{\fK_p}(G, \Ohm)(\bar{\IF}_p)$ and $\iota^\ad_s : X^\ad_p \times X^{\ad, p} \to \shS_{\fK^\ad_p}(G^\ad, \Ohm)(\bar{\IF}_p)$.

Set $W = W(\bar{\IF}_p)$ and $K_0 = W[1/p]$. Recall that we can identify $P^2_\cris(X/W)(1)$ with $ \bL_{\cris, t}$ by Prop.~\ref{compareCoh} and the isomorphism $H \tensor W \iso H^1_\cris(\sA_s/W)$ induces an isomorphism $\bL_{\cris, t} \iso L \tensor W$. Let $(M_p, \varphi) \subset P^2_\cris(X/W) \tensor K_0$ be the embedding in the statement of the theorem. There exists $g^\ad_p \in G^\ad(K_0)$ such that $g^\ad_p \cdot P^2_\cris(X/W) = M_p \subset P^2_\cris(X/W) \tensor K_0$. By Lem.~\ref{crysConj}, $g^\ad_p \in X^\ad_p$. 

Similarly, for every isometric embedding $M^p \subset P^2_\et(X, \IA_f^p)$ with $M^p \iso P^2_\et(X, \what{\IZ}^p)$, we can pick $g^{\ad, p} \in X^{\ad, p}$ such that $g^{\ad, p} \cdot P_\et^2(X, \what{\IZ}^p) = M^p$. Under the assumption $p > 18d + 4$, the period map $\wt{\Mod}_{2d, \fK^\ad_p, \IZ_{(p)}} \to \shS_{\fK^\ad_p}(G^\ad, \Ohm)$ is known to be surjective on $\bar{\IF}_p$-points (\cite[Thm~4.1]{Matsumoto}). It follows from Prop.~\ref{1.4.15} and Prop.~\ref{main} that any point in the preimage of $\iota_s^\ad((g^\ad_p, g^{\ad, p}))$ under the period map will give us the desired $(X', \xi')$. \qed
\subsection{Proof of Theorem~\ref{CML}}
\begin{proof}
Let $(X, \xi)$ be the quasi-polarized surface of Thm~\ref{CML}. Clearly we can assume that $\xi$ is primitive.  Let $t \in \wt{\Mod}_{2d, \fK^\ad_p, \IZ_{(p)}}(\bar{\IF}_p)$ be a point such that $(\sX_t, \bxi_t)$ can be identified with $(X, \xi)$. Let $s \in \shS_{\fK_p}(G, \Ohm)(\bar{\IF}_p)$ be a lift of $t$. By \cite[Thm~(0.4)]{Modp}, the isogeny class of $s$ contains a point $s'$ which lifts to a CM point on $\Sh_{\fK_p}(G, \Ohm)$. More precisely, there exists a finite flat extension $V'$ of $W := W(\bar{\IF}_p)$ and a $V'$-point $s'_{V'}$ lifting $s$ such that for some (and hence any) isomorphism $\bar{K}' \iso \IC$, where $K' := V[1/p]$, the complex point $s'_\IC := s'_{V'} \tensor \IC$ is a special point. Choose an $\bar{K}' \iso \IC$ and let $t'$ and $t'_\IC$ denote the images of $s'$ and $s'_\IC$ on $\shS_{\fK^\ad_p}(G^\ad, \Ohm)$ respectively. 

By the surjectivity of the period map over $\IC$, we can choose a preimage of $t'_\IC$ in $\wt{\Mod}_{2d, \fK^\ad_p, \IZ_{(p)}}(\IC)$, which we still denote by $t_\IC'$. The choice of the preimage will not be important. Let $(\sX_{t_\IC'}, \bxi_{t_\IC'})$ be the associated quasi-polarized complex K3 surface. By Thm~\ref{Zarhin}, $E:= \End_{\Hdg} T(\sX_{t'_\IC})_\IQ$ is a CM field generated by some Hodge isometry $\tau'$ as a $\IQ$-algebra, and we have $\dim_E T(\sX_{t'_\IC})_\IQ = 1$. We extend $\tau'$ to $\wt{\tau}' \in \End_{\Hdg}(P^2(\sX_{t'_\IC}, \IQ))$ such that $\wt{\tau}'$ fixes the Hodge classes. 

By lemma~\ref{surjection}, there exists a CSpin-isogeny $\psi'_\IC : \sA_{s_\IC'} \to \sA_{s_\IC'}$ which induces $\wt{\tau}'$ via the identifications
$$ P^2(\sX_{t'_\IC}, \IQ(1)) \stackrel{\sim}{\to} \bL_{B, t'_\IC} \tensor \IQ =  \bL_{B, s'_\IC} \tensor \IQ \into \End (H_{B, s'_\IC}) \tensor \IQ = \End (H^1(\sA_{s_\IC'}, \IQ)). $$ Now specialize $\psi'_\IC$ to a quasi-isogeny $\psi_{\tau'} : \sA_{s'} \to \sA_{s'}$. One easily checks that $\psi_{\tau'}$ is a CSpin-isogeny using the smooth and proper base change theorem and Rmk~\ref{dRham}. Let $\psi : \sA_{s} \to \sA_{s'}$ be a CSpin-isogeny connecting $s$ and $s'$. Set $\psi_{\tau} : = \psi^{-1} \circ \psi_{\tau'} \circ \psi$. $\psi_{\tau}$ induces an isometry of F-isocrystals $\tau_\cris: P^2_\cris(\sX_{t}/W)[1/p] \to P^2_\cris(\sX_{t}/W)[1/p]$. By Lem.~\ref{workhorce}, we can find a lift $t_V$ of $t$ for some finite flat extension $V$ over $W$ with fraction field $K = V[1/p]$ such that the filtration on $P^2_\cris(\sX_{t}/W) \tensor K$ induced by $(\sX_{t_V}, \bxi_{t_V})$ is respected by $\tau_\cris \tensor K$. We will show that $\sX_{t_{V}} \tensor_{V} \IC$ has CM for any $V \into \IC$. 

Let $s_V \in \shS_{\fK_p}(G, \Ohm)(V)$ be a lift of $t_V$ such that $s = s_V \tensor \bar{\IF}_p$. As in the proof of Thm~\ref{mainTh}, $\psi_{\tau}$ respects the Hodge filtration induced on $H^1_\cris(\sA_{s}/W) \tensor K$ by $\sA_{s_V}$.  By \cite[Thm~3.15]{BO} again, we may lift $\psi_{\tau}$ to a quasi-isogeny $\psi_{\tau, V} : \sA_{s_V} \to \sA_{s_V}$.  For any embedding $V \into \IC$, the CSpin-isogeny $\psi_{\tau, V} \tensor \IC: \sA_{s_V \tensor \IC} \to \sA_{s_V \tensor \IC}$ induces a Hodge isometry $\wt{\tau}: P^2(\sX_{t_V \tensor \IC}, \IQ) \to P^2(\sX_{t_V \tensor \IC}, \IQ)$. One may easily check from the construction that $\wt{\tau} \tensor \IQ_\ell$ is sent precisely to $\wt{\tau}' \tensor \IQ_\ell$ via the isomorphisms
\begin{equation}
\label{transfer}
P^2(\sX_{t_V \tensor \IC}, \IQ) \tensor \IQ_\ell \iso P^2_\et(\sX_t, \IQ_\ell) = \bL_{\ell, t} \stackrel{\text{conj. by }(\psi^*_\et)^{-1}}{\to} \bL_{\ell, t'} \iso P^2(\sX_{t'_\IC}, \IQ) \tensor \IQ_\ell.
\end{equation}
Therefore, we have an isomorphism of $\IQ$-algebras $\IQ(\wt{\tau}) \iso \IQ(\wt{\tau}')$. Moreover, we have a commutative diagram 
\begin{center}
\begin{tikzcd}
\<\bxi_{t_V \tensor \IC}\>^\perp  \arrow{r}{} \arrow{d}{} & \<\bxi_{t}\>^\perp \arrow{d}{} &  & \<\bxi_{t'_\IC} \>^\perp  \arrow{d}{} \\
 L(\sA_{s_\IC}) \arrow{r}{} & L(\sA_s) \arrow{r}{\text{conj. by }\psi} & L(\sA_{s'}) & L(\sA_{s'_\IC}) \arrow{l}{}
\end{tikzcd}
\end{center}
Here the orthogonal complements $\<\bxi_{t_V \tensor \IC}\>^\perp, \<\bxi_{t}\>^\perp$, and $\<\bxi_{t'_\IC} \>^\perp$ are taken inside $\Pic(\sX_{t_V \tensor \IC})$, $\Pic(\sX_{t})$, and $\Pic(\sX_{t'_\IC})$ respectively.
All arrows in the diagram are isomorphisms of quadratic lattices: The top horizontal arrow is an isomorphism because the construction in Lem.~\ref{workhorce} lifts the entire Picard group, the vertical arrows are isomorphisms given by Prop.~\ref{submotive}, and the specialization map $L(\sA_{s'_\IC}) \to L(\sA_{s'})$ is an isomorphism by Prop.~\ref{CMsp}. The diagram gives us an isometry $\<\bxi_{t_{V} \tensor \IC}\>^\perp  \stackrel{\sim}{\to}  \<\bxi_{t'_\IC} \>^\perp $ which is clearly compatible with (\ref{transfer}). Therefore, (\ref{transfer}) restricts to an isometry
$$ T(\sX_{t_{V} \tensor \IC}) \tensor \IQ_\ell \iso T(\sX_{t'_\IC}) \tensor \IQ_\ell.  $$
Let $\tau$ be the restriction of $\wt{\tau}$ to $T(\sX_{t_V \tensor \IC}, \IQ)$. Then the above isomorphism tells us that $\IQ(\wt{\tau}) \iso \IQ(\wt{\tau}')$ restricts to $\IQ(\tau) \iso \IQ(\tau')$. Since $\tau$ acts as an Hodge endomorphism on $T(\sX_{t_V \tensor \IC})$ and $\dim_{\IQ(\tau)} T(\sX_{t_{V} \tensor \IC})_\IQ = 1$, by Thm~\ref{Zarhin} $\sX_{t_V \tensor \IC}$ has CM. \end{proof}

Although we introduced the K3 surface $\sX_{t'_\IC}$ to make the argument more symmetric, one can also argue purely in terms of Hodge structures of K3 type. 

\section{Further Remarks on Isogenies}

Motivated by global Torelli theorems, we make the following definitions. 
\begin{definition}
Let $k$ be a perfect field with algebraic closure $\bar{k}$. Let $X, X'$ be K3 surfaces over $k$ and let $f : X \rightsquigarrow X'$ be an isogeny over $k$. 
\begin{itemize}
    \item We say that $f$ is \textit{polarizable} (resp. \textit{quasi-polarizable}) if there exists an ample (resp. big and nef) class $\xi \in \Pic(X')_\IQ$ such that $f^*(\xi)$ is still ample (resp. big and nef).
    \item If $\mathrm{char\,} k = 0$, we say $f$ is $\IZ$-integral if the induced isometries $H^2_\et(X_{\bar{k}}', \IQ_\ell) \stackrel{\sim}{\to} H^2_\et(X_{\bar{k}}, \IQ_\ell)$ preserves the $\IZ_\ell$-integral structures for every prime $\ell$.
    \item  If $\mathrm{char\,} k = p > 0$, we say $f$ is $\IZ$-integral if the induced map $H^2_\et(X_{\bar{k}}', \IQ_\ell) \stackrel{\sim}{\to} H^2_\et(X_{\bar{k}}, \IQ_\ell)$ preserves the $\IZ_\ell$-integral structures for every prime $\ell \neq p$ and $H^2_\cris(X'/W(k))[1/p] \stackrel{\sim}{\to} H^2_\cris(X/W(k))[1/p]$ preserves the $W(k)$-integral structures. 
\end{itemize}
\end{definition}

From now on, we let $k$ denote some algebraically closed field of characteristic $p \neq 2$.

If an isogeny $f : X \rightsquigarrow X'$ is equivalent (cf. Def.~\ref{isogdef}) to the graph of an isomorphism $\iota : X \stackrel{\sim}{\to} X'$, then we simply say that $f$ is induced by the isomorphism $\iota$. The classical Torelli theorem implies that when $k = \IC$, an isogeny $f : X \rightsquigarrow X'$ is induced by an isomorphism if and only if $f$ is polarizable and $\IZ$-integral. In fact, in view of Buskin's theorem, the classical Torelli theorem is equivalent to this statement about isogenies. We will explain that Ogus' crystalline Torelli theorem \cite[Thm~II]{Ogus2} can be formulated in terms of isogenies just as the classical Torelli theorem (cf. Thm~\ref{Ogusisog} below).

For abelian varieties there is no distinction between isogenies and polarizable isogenies, because the pullback of an ample divisor along an isogeny of abelian varieties is clearly always ample. For K3 surfaces however, there is a special class of isogenies customarily called ``reflection in (-2)-curves", which are never polarizable: Let $X$ be a K3 surface over $k$ and let $\beta$ be a line bundle on $X$ with $\beta^2 = -2$. Then up to replacing $\beta$ by $\beta^\vee$, $\beta$ is effective (cf. \cite[1.1.4]{K3book}). Let $C$ be a curve representing $\beta$ and $\sO_C$ be the structure sheaf of $C$, which we view as a coherent sheaf on $X$. Then $\sO_C$ is a spherical object in the bounded derived category of $X$, which we denote by $D(X)$. As a spherical object, $- \tensor \sO_C$ induces a Fourier-Mukai auto-equivalence $T_C : D(X) \stackrel{\sim}{\to} D(X)$. By Orlov's theorem  \cite[2.2]{Orlov}, $R_C$ is induced by a Fourier-Mukai kernel $P_C$, which is a perfect complex on $X \times X$. The Chow realization of the $P_C$ (cf. \cite[2.9]{LO1}) induces an action on the second cohomology which is nothing but reflection in $\beta$, i.e., $x \mapsto x + \< x, \beta \> \beta$, for every cohomology theory.\footnote{The action on the entire Mukai lattice by $P_C$ is reflection in the Mukai vector $\< 0, \beta, 1\>$ (cf. the proof of \cite[Prop.~6.2]{LO1}). This of course restricts to reflection in $\beta$ on second cohomology. One can also use $\sO_C(-1)$ instead as in \cite[10.3(iii)]{HuyB2} to obtain a reflection in $\<0, \beta, 0\>$ on the Mukai lattice.} We can view the Chow realization of $P_C$ as an isogeny $R_C : X \rightsquigarrow X$. One quickly checks from the formula $x \mapsto x + \< x, \beta \> \beta$ that $R_C$ is $\IZ$-integral.

It turns out that every isogeny differs from a polarizable isogeny by a composition of reflections in $(-2)$-curves up to a sign. To make this statement precise, view $R_C$ as an element of the $\IQ$-algebra of correspodences of degree $0$ from $X$ to $X$, where two correspondences are viewed as equivalent if their crystalline and $\ell$-adic realizations all agree. Let $R$ be the group of auto-isogenies $X \rightsquigarrow X$ generated by $\{ R_C : C \text{ is a $(-2)$-curve on $X$} \}$ in this correspondence algebra. Let $\pm R$ be the group generated by $R$ and $-1$. Then we have the following:

\begin{lemma}
\label{refl} Let $X, X'$ be K3 surfaces over $k$. 
\begin{enumerate}[label=\upshape{(\alph*)}]
    \item For every class $\xi \in \NS(X)$ with $\xi^2 > 0$, there exists an element $\alpha \in \pm R$ such that $\alpha^*(\xi)$ is big and nef. 
    \item For every isogeny $f : X \rightsquigarrow X'$, there exists an element $\alpha \in \pm R$ such that $f \circ \alpha$ is a polarizable isogeny. 
\end{enumerate}
\end{lemma}

\begin{proof}
(a) is a restatement of \cite[Lem.~7.9]{Ogus} in terms of isogenies. To prove (b), let $a' \in \NS(X')$ be an ample class and set $a := f^*(a') \in \NS(X)$. Take an open neighborhood $D_{a'}$ of $a'$ in $\NS(X')_\IR$ which is contained in the ample cone of $X'$. Since $f^*$ is an isometry, $a^2 > 0$. By (a) there exists an element $\alpha \in \pm R$ such that $\alpha^*(a)$ is big and nef. Since the nef cone is the closure of the ample cone (\cite[1.4 C, 2.2 B]{Laz}), $\alpha^*(f^*(D_{a'}))$ intersects the ample cone of $X$.
\end{proof}

\begin{remark}
The above result and arguments are well known to experts and are usually used as a reduction step (e.g., \cite[Prop.~6.2]{Buskin},  \cite[Lem.~6.2]{LO1}). 
\end{remark}

\begin{lemma}
\label{spread}
Let $X$ be a supersingular K3 surface over $k$. Then the maps 
\begin{enumerate}[label=(\alph*)]
    \item $c_1 : \NS(X) \tensor \IZ_\ell \to H^2_\et(X, \IZ_\ell)$ for every prime $\ell \neq p$;
    \item $c_1 : \NS(X) \tensor \IZ_p \to H^2_\cris(X/W(k))^{F = p}$
\end{enumerate}
are isomorphisms. 
\end{lemma}
\begin{proof} Let $\ell$ be any prime not equal to $p$. First, we show that $\NS(X) \tensor \IQ_\ell \to H^2_\et(X, \IQ_\ell)$ and $\NS(X) \tensor \IQ_p \to H^2_\cris(X/W(k))^{F = p}[1/p]$ are isomorphisms. If $k = \bar{\IF}_p$, this follows from the Tate conjecture for K3 surfaces \cite[Thm~1]{Keerthi}. One can then show this for a general algebraically closed $k$ by a theorem of Artin \cite[Thm~1.1]{Artin} and a standard spreading out argument. 

Now it suffices to show that the maps $\NS(X) \tensor \IZ_\ell \to H^2_\et(X, \IZ_\ell)$ and $\NS(X) \tensor \IZ_p \to H^2_\cris(X/W(k))^{F = p}$ have torsion-free cokernels. For the $\ell$-adic Chern class map, this follows from a Brauer group argument (cf. \cite[Lem.~2.2.2]{LMS} and its proof). For the cristalline Chern class map, this follows from \cite[Rmk~3.5]{Deligne2} (see also \cite[Lem.~2.2.4]{LMS}).
\end{proof}

\begin{theorem}
\label{Ogusisog}
An isogeny $f : X \rightsquigarrow X'$ between two supersingular K3 surfaces $X, X'$ over $k$ is induced by an isomorphism if and only if $f$ is polarizable and $\IZ$-integral. 
\end{theorem}
\begin{proof}
Clearly we only need to show the ``if'' direction.
By Lem.~\ref{spread}, $f^*$ induces an isomorphism $\NS(X') \to \NS(X)$ of quadratic lattices over $\IZ$. By \cite[Thm~II, II$''$]{Ogus2}, it suffices to show that $f^*$ maps the ample cone of $\NS(X')$ to that of $\NS(X)$. Define 
$$ V_X = \{ x \in \NS(X)_\IR : x^2 > 0 \text{ and } \< x, \beta\> \neq 0 \text{ for all } \beta^2 = -2 \} $$
and define $V_{X'}$ verbatim. Since $f^*(V_{X'}) = V_{X}$, and the ample cones of $X, X'$ are connected components of $V_{X}, V_{X'}$, it suffices to show that the preimage of the ample cone of $X'$ under $f^*$ intersects the ample cone $X$. This follows from our assumption that $f$ is polarizable.   
\end{proof}

By Lem.~\ref{corr} and Lem.~\ref{spread}, every isometry $\NS(X') \to \NS(X)$ comes from an isogeny. Therefore, the above theorem is a reformulation of \cite[Thm~II]{Ogus2}. 

\begin{remark}
We conjecture that Theorem~\ref{Ogusisog} holds for non-supersingular K3 surfaces as well. 
\end{remark}

Now we turn our attention to the formulation of the main theorem. An exact analogue of Theorem~\ref{teaser} in positive characteristic will be: 
\begin{conjecture}
\label{conj1}
Let $X$ be a K3 surface over $k$. Let $\Lambda_p$ (resp. $\Lambda^p$) be a quadratic lattice over $W(k)$ (resp. $\what{\IZ}^p$) which is abstractly isomorphic to $H^2_\cris(X/W(k))$ (resp. $H^2_\et(X, \what{\IZ}^p$)). Equip $\Lambda_p$ with a Frobenius action $\varphi$ such that $(\Lambda_p, \varphi)$ has the structure of a K3 crystal. Then for each pair of isometric embeddings $(\Lambda_p, \varphi) \subset (H^2_\cris(X/W(k))[1/p], F)$ and $\Lambda^p \subset H^2_\et(X, \IA_f^p)$, there exists another K3 surface $X'$ together with an isogeny $f : X \rightsquigarrow X'$ such that $f^* H^2_\cris(X'/W(k)) = \Lambda_p$ and $f^* H^2_\et(X', \what{\IZ}^p) = \Lambda^p$. 
\end{conjecture}

By ``$(\Lambda_p, \varphi) \subset (H^2_\cris(X/W(k))[1/p], F)$'' we just mean an isometric embedding $\Lambda_p \subset H^2_\cris(X/W(k))[1/p]$ such that $\varphi$ agrees with $F$. Below we state the quasi-polarized version of Conjecture~\ref{conj1} in terms of pointed lattices. A pointed lattice is a pair $(M, m)$ where $M$ is a quadratic lattice and $m \in M$ is a distinguished element. Morphisms between pointed lattices are defined in the obvious way.

\begin{conjecture}
\label{conj2}
Let $(X, \xi)$ be a quasi-polarized K3 surface over $k$. Let $(\Lambda_p, \lambda_p)$ (resp. $(\Lambda^p, \lambda^p)$) be a pointed lattice over $W(k)$ (resp. $\what{\IZ}^p$) such that $\Lambda_p$ (resp. $\Lambda^p$) is abstractly isomorphic to $H^2_\cris(X/W(k))$ (resp. $H^2_\et(X, \what{\IZ}^p$). Equip $\Lambda_p$ with a Frobenius action $\varphi$ such that $(\Lambda_p, \varphi)$ has the structure of a K3 crystal. Then for each pair of isometric embeddings $((\Lambda_p, \lambda_p), \varphi) \subset ((H^2_\cris(X/W(k))[1/p], c_1(\xi)), F)$ and $(\Lambda^p, \lambda^p) \subset (H^2_\et(X, \IA_f^p), c_1(\xi))$, there exists another quasi-polarized K3 surface $(X', \xi')$ together with an isogeny $f : X \rightsquigarrow X'$ such that $f^* (H^2_\cris(X'/W(k)), c_1(\xi')) = (\Lambda_p, \lambda_p)$ and $f^* (H^2_\et(X', \what{\IZ}^p), c_1(\xi')) = (\Lambda^p, \lambda^p)$. 
\end{conjecture}

\begin{proposition}
\label{equiv}
Conjecture~\ref{conj1} and Conjecture~\ref{conj2} are equivalent.
\end{proposition}
\begin{proof}
We first show that Conjecture~\ref{conj2} implies Conjecture~\ref{conj1}. Suppose that $(X, \Lambda_p \subset H^2_\cris(X/W(k))[1/p], \Lambda^p \subset H^2_\et(X, \what{\IZ}^p))$ is a tuple which satisfies the hypothesis of Conjecture~\ref{conj1}. Note that $\Lambda_p^{\varphi = p}$ is a $\IZ_p$-lattice in the $\IQ_p$-vector space $(H^2_\cris(X/W(k))^{F = p}) \tensor \IQ_p$. 
Let $\xi \in \NS(X)$ be the class of any quasi-polarization. Up to replacing $\xi$ by a $\IZ$-multiple, we can assume that its crystalline Chern class is $\lambda_p$ for some $\lambda_p \in \Lambda_p^{\varphi = p}$ and its \'etale Chern class is some $\lambda^p$ for some $\lambda^p \in \Lambda^p$. We get the desired $X'$ and $f : X \rightsquigarrow X'$ by applying Conjecture~\ref{conj2} to $(X, (\Lambda_p, \lambda_p), (\Lambda^p, \lambda^p))$. 

Conversely, Conjecture~\ref{conj1} implies Conjecture~\ref{conj2} by Lemma~\ref{refl}(a). 
\end{proof}

\begin{remark}
\label{mainRemark}
Let $2d = \xi^2$. Under the assumptions $k = \bar{\IF}_p$ and $p > 18 d + 4$, Theorem~\ref{mainCor} proves the $p$-part and the prime-to-$2d$ part of Conjecture~\ref{conj2}, i.e., Conjecture~\ref{conj2} with $\IA^p_f$ replaced by the restricted product $$\prod_{\ell \nmid 2dp}(\IQ_\ell : \IZ_\ell) = \{ (a_\ell) \in \prod_{\ell \nmid 2dp}\IQ_\ell : a_\ell \in \IZ_\ell \text{ for all but finitely many } \ell \nmid 2dp \}.$$ The reason is that when $\ell \nmid 2dp$, $H^2_\et(X, \IZ_\ell) = P^2_\et(X, \IZ_\ell) \oplus c_1(\xi)$ as quadratic lattices. If $\ell \mid 2d$ and $\ell \neq p$, then $P^2_\et(X, \IZ_\ell) \oplus c_1(\xi)$ embeds as a finite index sublattice of $H^2_\et(X, \IZ_\ell)$. Therefore, even if it were not for the restriction $p > 18 d + 4$, Theorem~\ref{mainCor} is slightly weaker than Conjecture~\ref{conj2}. We also remark that if $p \mid 2d$, $P^2_\cris(X/W)$ is not a K3 crystal because it is no longer self-dual. Therefore, giving an anologue of Theorem~\ref{mainCor} for $p \mid d$ will entail a different formulation. However, our main interest is in Conjecture~\ref{conj2}, whose formulation works whether or not $p \mid 2d$.  
\end{remark}

\begin{remark}
We do not expect Conjecture~\ref{conj2} to hold if quasi-polarizations are replaced by polarizations. The reason is that $f$ does not preserve the integral structures of N\'eron-Severi groups in general, so it does not have to pull back $V_{X'}$ to $V_X$ as in the proof of Thm~\ref{Ogusisog}. Nonetheless, $f$ certainly does pull back the closure of $V_{X'}$ to that of $V_X$. As we have discussed in the proof of Lem.~\ref{refl}, for any smooth projective variety, a big and nef class is a class which has positive self-intersection number and lies in the closure of the ample cone. This partially explains the usefulness of considering big and nef classes. 
\end{remark}

Finally, we explain that the surjectivity statement for Ogus' crystalline period map \cite[Prop.~1.16]{Ogus2} and Lem.~\ref{spread} yield a complete proof of Conjecture~\ref{conj1} in the supersingular case: 

\begin{proposition}
\label{ss+}
Conjecture~\ref{conj1} holds for $X$ supersingular. 
\end{proposition}
\begin{proof}
Let $N^p := \Lambda^p$, $N_p := \Lambda_p^{\varphi = p}$, and $\what{N}:= N^p \times N_p$. By Lem.~\ref{spread}, the embedding $(\Lambda_p, \varphi) \subset (H^2_\cris(X/W(k))[1/p], F)$ restricts to an embedding $N_p \subset \NS(X) \tensor \IQ_p$. Similarly the embedding $\Lambda^p = N^p \subset H^2_\et(X, \IA_f^p) \stackrel{c_1}{\iso} \NS(X) \tensor \what{\IZ}^p$. Together we obtain an embedding $\what{N} \subset \NS(X) \tensor \IA_f$. Set $N = \what{N} \cap \NS(X)_\IQ$. Then $N$ is quadratic lattice over $\IZ$ such that the natural maps $N \tensor \IZ_p \to N_p$, $N \tensor \what{\IZ}^p \to N^p$ and $N \tensor \IQ \to \NS(X)_\IQ$ induced by inclusions are isomorphisms of quadratic lattices. 

We check that $N$ is a K3 lattice in the sense of \cite[Def.~1.6]{Ogus2}. First, $N$ is even, as $N \tensor \IZ_2 \iso \NS(X) \tensor \IZ_2$. Since $N \tensor \IQ \iso \NS(X) \tensor \IQ$, (a) and (b) in \cite[Def.~1.6]{Ogus2} are clearly satisfied. We only need to check that the cokernel of the embedding $N \to N^\vee$ induced by the symmetric bilinear form on $N$ is annihilated by $p$. Note that $(\Lambda_p, \varphi)$ is a supersingular K3 crystal and $N \tensor \IZ_p = N_p$ is the Tate module of $(\Lambda_p, \varphi)$ (cf. \cite[3.1, 3.2]{Ogus}). That $N^\vee/ N$ is annihilated by $p$ follows from \cite[3.13, 3.14]{Ogus}. 

Let $\sM_N$ be the $k$-scheme parametrizing the characteristic subspaces of $(N^\vee/N) \tensor_{\IF_p} k$ (see \cite[Sect.~4]{Ogus}). It is the moduli space of $N$-rigidified K3 crystals.\footnote{The reader can look at \cite[4.4, 5.2]{LiedtkeExp} for an exposition of the results in \cite[Sect.~4]{Ogus} and the period map.} By \cite[Prop.~1.16]{Ogus2}, the natural period map from the moduli space of $N$-marked supersingular K3 surfaces (cf. \cite[Thm~2.7]{Ogus2}) to $\sM_N$ is surjective. Therefore, there exists another K3 surface $X'$ together with an isomorphism $N \stackrel{\sim}{\to} \NS(X')$ which fits into a commutative diagram 
\begin{center}
    \begin{tikzcd}
    N \arrow{r}{i} \arrow{d}{} & \NS(X') \arrow{d}{c_1} \\
    (\Lambda_p, \varphi) \arrow{r}{\sim} & (H^2_\cris(X'/W(k)), F)
    \end{tikzcd}.
\end{center}
Note that the bottom arrow is completely determined by the top arrow, as the maps $N \tensor_\IZ K_0 \to \Lambda_p [1/p]$ and $\NS(X') \tensor_\IZ W(k)[1/p] \stackrel{c_1 \tensor 1}{\to} H^2_\cris(X'/W(k))[1/p]$ are isomorphisms. By composing the identification $N \tensor \IQ = \NS(X)_\IQ$ given by the inclusion $N \subset \NS(X)_\IQ$ with $i^{-1} \tensor \IQ$ we get an isometry $f : \NS(X')_\IQ \stackrel{\sim}{\to} \NS(X)_\IQ$ such that $\NS(X')$ is sent to $N$ and $H^2_\cris(X'/W(k))$ is sent to $\Lambda_p$. By Lem.~\ref{corr}, $f$ is induced by an isogeny $X \rightsquigarrow X'$. $X'$ together with this isogeny is what we seek.
\end{proof}

\begin{remark}
In \cite{Ogus2}, Ogus gave a slightly different definition of K3 crystals. In particular, he asked the crystalline discriminant to be $-1$ (cf. \cite[Def.~1.4]{Ogus2}). This additional requirement does not affect the formulation of Conjecture~\ref{conj1} and~\ref{conj2}, because if a K3 crystal $(\Lambda_p, \varphi)$ embeds isometrically into $(H^2_\cris(X/W(k))[1/p], F)$ for any K3 surface $X$, then $(\Lambda_p, \varphi)$ necessarily has crystalline discriminant $-1$. 
\end{remark}

\begin{remark}
\label{rmkss}
We remark that in fact any two supersingular K3 surfaces $X, X'$ over $k$ are isogenous by Lem.~\ref{corr} and Lem.~\ref{spread}, as $\NS(X)_\IQ \iso \NS(X')_\IQ$ as $\IQ$-quadratic lattices. By works of Artin and Rudakov-Shafarevich (see \cite[Prop.~17.2.19, 17.2.20]{K3book}), $\NS(X) \iso N_{p, \sigma}$ for some $1 \le \sigma \le 10$, where $N_{p, \sigma}$ is the unique even, non-degenerate $\IZ$-lattice with signature $(1, 21)$ and discriminant group $(\IZ/p\IZ)^{2 \sigma}$. One may check that $N_{p, \sigma} \tensor \IQ \iso N_{p, \sigma'} \tensor \IQ$ for any $\sigma, \sigma'$ by the Hasse principle: $N_{p, \sigma} \tensor \IR \iso N_{p, \sigma'} \tensor \IR$ as a real quadratic form is determined by its signature. For any $\ell \neq p$, we see that $\NS(X) \tensor \IQ_\ell = \Lambda \tensor \IQ_\ell$ by a lifting argument. Finally, $N_{p, \sigma} \tensor \IQ_p \iso N_{p, \sigma'} \tensor \IQ_p$ as they have the same discriminant and Hasse invariant. 
\end{remark}

\paragraph{Acknowledgments} 
First and foremost I need to thank my advisor Mark Kisin for suggesting this problem to me and supporting me through the project. 
It is also a pleasure to thank F. Charles, L. Chen, J. Lam, Q. Li, T. Nie, L. Mocz, A. Petrov and A. Shankar for helpful conversations. I thank K. Madapusi Pera for clarifying several details about his work, C. Schoen for reading a previous version of this paper and L. Mocz for giving many helpful comments. I am deeply grateful for the two anonymous referees who pointed out many inaccuracies in an earlier version of the paper and offered suggestions which greatly improved the text. 

{\footnotesize
\bibliography{sample}}

\bibliographystyle{abbrv}

\end{document}